\documentclass[12pt]{article}
\usepackage[a4paper, margin=1.1in]{geometry}
\usepackage{amsmath,amsthm}
\usepackage{amsfonts,amssymb}
\usepackage{graphicx}
\usepackage{breqn}
\usepackage{amsmath}
\usepackage{float}
\usepackage{caption}
\usepackage{chngcntr}
\usepackage{fncylab}
\usepackage{verbatim}
\usepackage{cite}
\theoremstyle{plain} 

\newtheorem{thm}{Theorem}[section]
\theoremstyle{definition}

\newtheorem{op}{Operation}[section]
\newtheorem{cor}[thm]{Corollary}
\newtheorem{lem}[thm]{Lemma}

\newtheorem{defn}{Definition}[section]
\theoremstyle{remark}

\newtheorem{exam}{Example}[section]

\begin{document}

\title
{\bf{Spectra of new graph operations based on central graph}}
\author {\small Jahfar T K \footnote{jahfartk@gmail.com} and Chithra A V \footnote{chithra@nitc.ac.in} \\ \small Department of Mathematics, National Institute of Technology, Calicut, Kerala, India-673601}
\date{ }
\maketitle
\begin{abstract}
In this paper, we introduce central vertex corona, central edge corona, and central edge neighborhood corona of graphs using central graph. Also,  we determine their adjacency spectrum, Laplacian spectrum and signless Laplacian spectrum. From our results, it is possible to obtain infinitely many pairs of adjacency (respectively, Laplacian and signless Laplacian) cospectral graphs.   As an application, we calculate the number of spanning trees and the  Kirchhoff index of the resulting graphs.
\end{abstract}

\hspace{-0.6cm}\textbf{AMS classification}: 05C50
\section{Introduction}
In this paper, we consider only simple connected graphs.  Let $G=(V(G),E(G))$ be a graph with vertex set $V(G)=\{v_{1}, v_{2},...,v_{n}\}$ and edge set $E(G)=\{e_{1}, e_{2},...,e_{m}\}$, and let $d_{i}$ be the degree of the vertex $v_{i}, v_i\in V(G).$ The adjacency matrix $A(G)$ of the graph $G$ is a square symmetric matrix of order $n$ whose $(i,j)^{th}$ entry is equal to unity if the vertices $v_{i}$ and $v_{j}$  are adjacent, and is equal to zero otherwise. The Laplacian matrix of $G$, denoted by $L(G)$ is defined as $L(G)=D(G)-A(G)$ and the signless Laplacian matrix of $G$, denoted by $Q(G)$ is defined as $Q(G)=D(G)+A(G)$, where $D(G)$ is the diagonal matrix whose entries are the degrees of the vertices of $G$. The characteristic polynomial of the $n\times n$ matrix $M$ of $G$ is defined as $f(M,x)=|xI_n-M|$, where $I_n$ is identity matrix of order $n$. The matrices $A(G),L(G)$ and $Q(G)$ are real symmetric matrices, its eigenvalues are real. Denote the eigenvalues of $A(G),L(G)$ and $Q(G)$ by  $\lambda_1(G)\geqslant\lambda_2(G)\geqslant...\geqslant\lambda_n(G),0=\mu_1(G)<\mu_2(G)\leq...\leq \mu_n(G)$ and $\gamma_1(G)\geq\gamma_2(G)\geq...\geq\gamma_n(G)$, respectively. The set of all eigenvalues of $A(G)$ (respectively, $L(G)$, Q(G)) together with their multiplicities are called $A$-spectrum (respectively, $L$-spectrum, $Q$-spectrum) of $G$. Two non isomorphic graphs are said to be $A$-cospectral (respectively, $L$-cospectral, $Q$-cospectral), if they have the same $A$-spectrum (respectively, $L$-spectrum, $Q$-spectrum). Otherwise, they are non $A$-cospectral (respectively, non $L$-cospectral, non $Q$-cospectral) graphs. Let $G$ be a connected graph on $n$ vertices, then the number of spanning trees of $G$ is $t(G) = \frac{\mu_2(G)\cdot\mu_3(G)...\mu_n(G)}{n}$\cite{godsil2001algebraic}. In 1993, Klein and Randi{\'{c}} \cite{klein1993resistance} introduced the resistance distance $r_{i,j}^*$ between two vertices $v_i$ and $v_j$ in $G$ and is defined as the effective resistance between them when unit resistors are distributed on every
edge of G. The Kirchhoff index \cite{chen2007resistance} of a connected graph $G$ is defined   as $$Kf(G)=\sum_{i<j}^{}r_{i,j}^*,$$ where $r_{i,j}^*$ denotes the resistance distance between vertices $v_i$ and $v_j$ in  $G$.  In \cite{gutman1996quasi,zhu1996extensions}, the authors proved that
the Kirchhoff index of a graph $G$ is    $$Kf(G)=n\displaystyle\sum_{i=2}^{n}\frac{1}{\mu_i(G)}.$$
%
  In literature, there are many graph operations like, complement, disjoint union, join, cartesian product, direct product, strong product, lexicographic product, corona, edge corona, neighbourhood corona etc. The corona of two graphs first introduced by F.Harary and R.Frucht in \cite{harary1970corona}.
   Recently, several variants of corona product of two graphs have been introduced and their spectra are computed. In \cite{liu2013spectra}, Liu and Lu introduced subdivision-vertex and subdivision-edge neighbourhood  corona of two graphs and provided a complete description of their spectra. In \cite{lan2015spectra}, Lan and Zhou introduced $R$-vertex corona, $R$-edge corona, $R$-vertex neighborhood corona and $R$-edge neighborhood corona, and studied their spectra.
   Recently in \cite{adiga2018spectra,adiga2015spectra}, Adiga et.al introduced duplication corona, duplication neighborhood corona, duplication edge corona, N-vertex corona and N-edge corona of two regular graphs. Motivated by the above works, we define three new graph operations based on central graph.\\ 
  \par  The Kronecker product $A\otimes B$ of two matrices $A=(a_{i,j})_{m\times n}$ and $B=(b_{i,j})_{p\times q}$ is the $mp\times nq$ matrix obtained from A by replacing each entry $a_{i,j}$ by $a_{i,j}B$. For matrices A,B,C and D such that products AC and BD exist then $(A\otimes B)(C\otimes D)=AC\otimes BD.$ Also, if $A$ and $B$ are  invertible matrices, then $(A\otimes B)^{-1}=A^{-1}\otimes B^{-1}$,  $(A\otimes B)^{T}=A^{T}\otimes B^{T}$. If A and B are $n\times n$ and $p\times p$ matrices respectively, then $det(A\otimes B)=(detA)^p(detB)^n$\textnormal{\cite{horn1991topics}}. The incidence matrix of a graph $G$, $I(G)$ is the $n\times m$ matrix whose $(i,j)^{th}$ entry is 1 if $v_i$ is incident to $e_j$ and 0 otherwise. 
   It is  known \textnormal{\cite{cvetkovic1980spectra}} that,  $I(G)I(G)^T=A(G)+D(G)$ and if $G$ is an $r$-regular graph, then  $I(G)I(G)^T=A(G)+rI_n$. The \textit{complement} $\bar{G}$ of a graph $G$ is the graph with vertex set $V(G)$ and two vertices are adjacent in $\bar{G}$ if and only if they are not adjacent in $G$. The adjacency matrix of the complement of a graph $G$ is $A(\bar{G}) =J_n-I_n-A(G)$. Throughout this article, $I_n$ denote the identity matrix of order $n$. The symbol $O_{m\times n}$ and $J_n$ (respectively, $J_{1\times n}$ ) will stand for $m\times n$ and $n\times n$  matrices (respectively, $1\times n$ column vector) consisting of 0's and 1's. As usual, we denote a complete graph on $n$ vertices by $K_n$, a complete bipartite graph on $n=p+q$ vertices by $K_{p,q}$ and a path on $n$ vertices by $P_n$ .
 
  \par The rest of the paper is organized as follows. In Section 2, we present some definitions and lemmas that will be used later. In Section 3, we introduce a new graph operation, called central vertex corona and calculate its adjacency (respectively, Laplacian and signless Laplacian) spectrum. The number of spanning trees and the Kirchhoff index of the resulting graphs are calculated. Also, our results show how to construct non isomorphic cospectral families of graphs. In Section 4, we introduce a new graph operation, called central edge corona and calculate its adjacency (respectively, Laplacian and signless Laplacian) spectrum. Moreover, the formula for the spanning trees and the Kirchhoff index of the resulting graphs are obtained. Further, we obtain  some cospectral graphs.
    In Section 5, we define a new graph operation, called central edge neighborhood corona and calculate its adjacency (respectively, Laplacian and signless Laplacian) spectrum. Also, we formulate the number of spanning trees and the Kirchhoff index of edge neighborhood corona. Finally, as an application of these results we construct many pairs of non isomorphic cospectral graphs.
  
 \section{Preliminaries}
 In this section, we recall some definitions and results which will be useful to prove  our main results. 
\begin{lem}\textnormal{\cite{cvetkovic1980spectra}}
	Let $U,V,W$ and $X$ be matrices with $U$  invertible. Let \[S=\begin{pmatrix}
	U&V\\
	W&X
	\end{pmatrix}. 
	\]  Then $det(S)=det(U)det(X-WU^{-1}V)$.\\ If $X$ is invertible,  then $det(S)=det(X)det(U-VX^{-1}W).$ \\If $U$ and $W$ are commutes, then $det(S)=det(UX-WV)$.
\end{lem}
\begin{lem}\textnormal{\cite{cvetkovic1980spectra}}
	Let $G$ be a connected $r$-regular graph on $n$ vertices with an adjacency matrix $A$ having $t$ distinct eigenvalues $r=\lambda_1,\lambda_2,...,\lambda_t$.  Then there exists a polynomial \[P(x)=n\dfrac{(x-\lambda_2)(x-\lambda_3)...(x-\lambda_t)}{(r-\lambda_2)(r-\lambda_3)...(r-\lambda_t)}.\] such that $P(A)=J_n, P(r)=n$ and $P(\lambda_i) =0$ for $\lambda_i\neq r.$
\end{lem}
\begin{defn}\textnormal{\cite{das2018spectra}}
	The $M$-coronal $\chi_M(x)$  of $n\times n$ matrix $M$ is defined as the sum of the entries of the matrix $(xI_n-M)^{-1}$ $($if exists$)$, that is, $$\chi_M(x)=J_{n\times 1}^T(xI_n-M)^{-1}J_{n\times 1},$$ where $J_{n\times 1}$ denotes the column vector of size $n$ with all entries equal to one.
\end{defn}
\begin{lem}\textnormal{\cite{mcleman2011spectra}}
	Let $G$ be an $r$-regular graph on $n$  vertices.  Then $\chi_{A(G)}(x)=\frac{n}{x-r}$ .
\end{lem}
For Laplacian matrix each row sum is zero, so 
 \begin{align}
\chi_{L(G)}(x)=\frac{n}{x}.
\end{align}
\\Let $G$ be an $r$-regular graph on $n$  vertices. Then 
\begin{align}
\chi_{Q(G)}(x)=\frac{n}{x-2r}.
\end{align} 
\begin{lem}\textnormal{\cite{mcleman2011spectra}}
	Let $G$  be a bipartite graph $K_{p,q}$  with $p+q$ vertices. Then $\chi_{A(G)}(x)=\frac{(p+q)x+2pq}{(x^2-pq)}.$ 
\end{lem}    
\begin{lem}\textnormal{\cite{liu19}}
	Let $A$  be an $n\times n$ real matrix. Then $det(A+\alpha J_{n})=det(A)+\alpha J_{n\times 1}^T adj(A)J_{n\times 1}$, where $\alpha $  is a real number and $adj(A)$  is the adjoint of $A$. 
\end{lem}
\begin{cor}\textnormal{\cite{liu19}}
	Let $A$  be an $n\times n$  real matrix. Then \[det(xI_n-A-\alpha J_{n})=(1-\alpha \chi_A(x))det(xI_n-A).\] 
	
\end{cor}

 \begin{defn}\textnormal{\cite{vivin2008harmonious}}
	
	Let $G$ be a simple graph with $n$ vertices and $m$ edges. The central graph of $G$, denoted by $C(G)$  is obtained by sub dividing each edge of $G$ exactly once and joining all the nonadjacent vertices in $G$.\\ The graph $C(G)$ has $m+n$ vertices and $m+\frac{n(n-1)}{2}$ edges.\\ Let $\tilde{V}(G)=V(C(G))-V(G)$, be the set of vertices in $C(G)$ corresponding to the edges of $G$.
\end{defn}
\section{Spectra of central vertex corona of graphs}
In this section, we define a new corona operation on graphs and calculate their adjacency spectrum, Laplacian spectrum and signless Laplacian spectrum. Moreover, our results allows to construct cospectral graphs. The Kirchhoff index and the number of spanning trees of the resulting graphs are obtained. 
\begin{defn} 	Let $G_i$ be a graph with $n_i$ vertices and $m_i$ edges for  $i=1,2$. Then the central vertex corona $G_1 \odot G_2$ of two graphs $G_1$ and $G_2$ is the graph obtained by taking one copy of $C(G_1)$ and $|V(G_1)|$  copies of $G_2$ and joining the vertex $v_i$ of $G_1$ to every vertex in the $i^{th}$ copy of $G_2$. \\The adjacency matrix of $G_1 \odot G_2$ can be written as
	\begin{align*}
	A(G_1 \odot G_2)=\begin{pmatrix}
	A(\overline{G_1})&I(G_1)&I_{n_1}\otimes J_{1\times n_2} \\
	I(G_1)^T&O_{m_1\times m_1}& O_{m_1\times n_1n_2}\\
	I_{n_1}\otimes J_{1\times n_2}^T &O_{n_1n_2\times m_1}&I_{n_1}\otimes A(G_2) 
	\end{pmatrix}. 
	\end{align*}
	The graph  $G_1 \odot G_2$ has  $n_1+m_1+n_1n_2$ vertices and $m_1+\frac{n_1(n_1-1)}{2}+n_1m_2+n_1n_2$ edges.
	\begin{exam}
		Let $G_1=P_3$ and $G_2=P_2$. Then the two central vertex coronas $G_1 \odot G_2$ and $G_2 \odot G_1$ are depicted in Figure:1. 
	\end{exam}
		\begin{figure}[H]
		\begin{minipage}[b]{0.5\linewidth}
			\centering
			\includegraphics[width=8.0 cm]{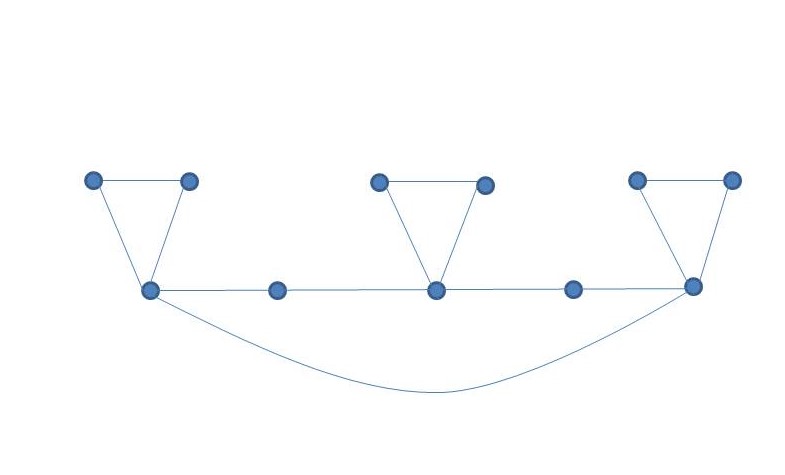}
			\caption{  $P_3 \odot P_2$}
			\label{pict12.jpg}
		\end{minipage}
		\hspace{0.5cm}
		\begin{minipage}[b]{0.4\linewidth}
			\centering
			\includegraphics[width=8.0 cm]{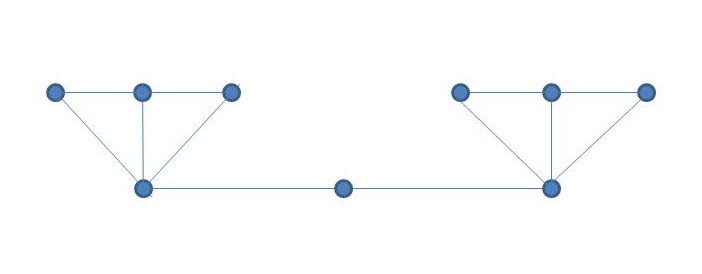}
			\caption{  $P_2 \odot  P_3$}
			\label{pict13.jpg}
		\end{minipage}
	\caption{Figure:1 An example of central vertex corona graphs}
	\end{figure}	
\end{defn}
	\begin{thm}
	Let $G_1$ be an $r_1$-regular graph with $n_1$ vertices and $m_1$ edges and $G_2$ be an arbitrary graph with $n_2$ vertices. Then the  adjacency characteristic polynomial of $G_1 \odot G_2$ is
	\begin{equation*}
		\begin{aligned}
			f(A(G_1 \odot G_2),x) =&x^{m_1-n_1}\prod_{j=1}^{n_2}(x-\lambda_j(G_2))^{n_1}\\& \times\prod_{i=1}^{n_1}\left(x^2+x-x\chi_{A{(G_2)}}(x)-r_1-xP(\lambda_i(G_1))+(x-1)\lambda_i(G_1)\right).\\ 
		\end{aligned}
	\end{equation*} 

\end{thm}
\begin{proof}
 The characteristic polynomial of $G_1 \odot G_2$ is \\
\begin{align*}
f(A(G_1 \odot G_2),x) =&det(xI_{n_1+m_1+n_1n_2}-A(G_1 \odot G_2))\\=&  det\left (\begin{matrix}
xI_{n_1}-A(\bar G_1)&-I(G_1)&-I_{n_1}\otimes J_{1\times n_2} \\
-I(G_1)^T&xI_{m_1}& O_{m_1\times n_1n_2}\\
-I_{n_1}\otimes J_{1\times n_2}^T&O_{ n_1n_2\times m_1}&I_{n_1}\otimes (xI_{n_2}-A(G_2)) 
\end{matrix} \right).
\end{align*}
\\ \textnormal{By Lemmas 2.1, 2.2, Definition 2.1 and Corollary 2.6, we have}\\
\begin{align*}
f(A(G_1 \odot G_2),x)=&\det(I_{n_1}\otimes (xI_{n_2}-A(G_2))\det S,\\
\textnormal{where~S} =&\begin{pmatrix}
xI_{n_1}-J_{n_1}+I_{n_1}+A( G_1)&{-I(G_1)}\\
-I(G_1)^T&{xI_m}_1\\
\end{pmatrix}\\\hspace{-2cm} &-\begin{pmatrix}
- I_{n_1}\otimes J_{1\times n_2} \\
O_{m_1\times n_1n_2}
\end{pmatrix}\left(I_{n_1}\otimes (xI_{n_2}-A(G_2))\right)^{-1}\begin{pmatrix}
\-{-I_{n_1}\otimes J_{1\times n_2}^T}&O_{ n_1n_2\times m_1}\end {pmatrix}\\
=&\begin{pmatrix}
xI_{n_1}-J_{n_1}+I_{n_1}+A( G_1)-\chi_{A{(G_2)}}(x)I_{n_1}&-I(G_1)\\
-I(G_1)^T&xI_{m_1}
\end{pmatrix}.\end{align*}Therefore,\begin{align*}
\det S=&det\left (\begin{matrix}
xI_{n_1}-J_{n_1}+I_{n_1}+A( G_1)-\chi_{A{(G_2)}}(x)I_{n_1}&-I(G_1)\\
-I(G_1)^T&{xI_m}_1
\end{matrix} \right)\\
= & x^{m_1} det\left(xI_{n_1}-J_{n_1}+I_{n_1}+A( G_1)-\chi_{A{(G_2)}}(x)I_{n_1}-\frac{I(G_1)I(G_1)^T}{x}\right)\\ 
= & x^{m_1} det\left((x+1-\chi_{A{(G_2)}}(x))I_{n_1}-J_{n_1}+A( G_1)-\frac{I(G_1)I(G_1)^T}{x}\right)\\ 
= & x^{m_1} det\left((x+1-\chi_{A{(G_2)}}(x)-\frac{r_1}{x})I_{n_1}-J_{n_1}+(1-\frac{1}{x})A( G_1)\right)\\ 
=&x^{m_1} \prod_{i=1}^{n_1}\left(x+1-\chi_{A{(G_2)}}(x)-\frac{r_1}{x}-P(\lambda_i(G_1))+(1-\frac{1}{x})\lambda_i(G_1)\right).\\ 
\end{align*}
Therefore the adjacency characteristic polynomial of $G_1 \odot G_2$ is	
\begin{equation*}
	\begin{aligned}
		f(A(G_1 \odot G_2),x)=&\prod_{j=1}^{n_2}(x-\lambda_j(G_2))^{n_1}\det S\\ =&x^{m_1-n_1}\prod_{j=1}^{n_2}(x-\lambda_j(G_2))^{n_1}\\ &\times\prod_{i=1}^{n_1}\left(x^2+x-x\chi_{A{(G_2)}}(x)-r_1-xP(\lambda_i(G_1))+(x-1)\lambda_i(G_1)\right).\\ 
	\end{aligned}
\end{equation*} 	 
\end{proof}	
 The following corollary  obtained from  Theorem 3.1 when $G_1$ and $G_2$ are both regular graphs.
 \begin{cor}
 Let $G_i$ be an $r_i$-regular graph with $n_i$ vertices and $m_i$ edges for ${i=1,2}.$ Then the adjacency characteristic polynomial of $G_1 \odot G_2$ is
 \begin{equation*}
 \begin{aligned}
 f(A(G_1 \odot G_2),x) =&x^{m_1-n_1}\left(x^3-x^2(r_2-1-r_1+n_1)-x(r_2+n_2+2r_1+r_1r_2-n_1r_2)+2r_1r_2\right)\\&\prod_{j=2}^{n_2}(x-\lambda_j(G_2))^{n_1} \prod_{i=2}^{n_1}\bigg[x^3-x^2(r_2-1-\lambda_i(G_1))-x(r_2+n_2+r_1+\\&~~~~~~~~~~~~~~~~~~~~~~~~~~~~~~~~~r_2\lambda_i(G_1)+\lambda_i(G_1))+r_1r_2+r_2\lambda_i(G_1)\bigg].\\ 
 \end{aligned}
 \end{equation*} 	
 \end{cor} 
 \par The following corollary describes the complete spectrum of $G_1 \odot G_2$, when $G_1$ and $G_2$ are both regular graphs.
\begin{cor}
Let $G_i$ be an $r_i$-regular graph with $n_i$ vertices and $m_i$ edges for ${i=1,2}.$ Then the adjacency spectrum of $G_1 \odot G_2$ consists of 
\begin{enumerate}
	\item $0$, repeated $m_1-n_1$ times,
	\item $\lambda_j(G_2), $ repeated $n_1$ times for $j=2,3,...,n_2$,
	\item three roots of the equation $x^3-x^2(r_2-1-\lambda_i(G_1))-x(r_2+n_2+r_1+r_2\lambda_i(G_1)+\lambda_i(G_1))+r_1r_2+r_2\lambda_i(G_1)=0 $ for $i= 2,...,n_1,$
	\item three roots of the equation $x^3-x^2(r_2-1-r_1+n_1)-x(r_2+n_2+2r_1+r_1r_2-n_1r_2)+2r_1r_2=0.$  
\end{enumerate}
\end{cor}	 

\begin{exam}
	Let $G_1=K_{3,3}$ and $G_2=K_2$. Then the adjacency  eigenvalues of $G_1$ are 0 (multiplicity $4$)  $\pm 3$, eigenvalues of $G_2$ are $\pm 1$. Therefore, the adjacency  eigenvalues of $G_1 \odot G_2$ are 0 (multiplicity $6$), $-1$ (multiplicity $3$),  roots of the equation $x^3-6x+3=0$ (each root with multiplicity $4$), roots of the equation $x^3-3x^2=0$ and roots of the equation $x^3-3x^2-6x+6=0$ .\\
\end{exam}
 Next, we shall consider the adjacency spectrum of $G_1 \odot G_2$ when $G_1$ is a regular graph and $G_2=K_{p,q}$,  $G_2$ is non-regular if $p\ne q$.
\begin{cor}
Let $G_1$ be an $r_1$-regular graph with $n_1$ vertices and $m_1$ edges . Then the adjacency spectrum of  $ G_1\odot K_{p,q}$ consists of
\begin{enumerate}
	\item  $0$,  repeated $m_1-n_1+n_1(p+q-2)$ times,
	
	\item four roots of the equation $x^4+(1+\lambda_i(G_1))x^3-(pq+p+q+r_1+\lambda_i(G_1))x^2+(-3pq-pq\lambda_i(G_1))x+pq\lambda_i(G_1)+r_1pq=0$ for $ i=2,...,n_1,$
	\item four roots of the equation $x^4+(1+r_1-n_1)x^3-(pq+p+q+2r_1)x^2+(-3pq-pqr_1+pqn_1)x+2pqr_1=0.$
\end{enumerate}
\end{cor} 
Corollary 3.2 helps us to construct infinitely many pairs of $A$-cospectral graphs.
\begin{cor}
$(a)$ Let $G_1$ and $G_2$ be $A$-cospectral regular graphs and $H$ is any regular graph. Then $ H\odot G_1$ and $ H\odot G_2$ are $A$-cospectral.\\
$(b)$
Let $G_1$ and $G_2$ be $A$-cospectral regular graphs and $H$ is any regular graph. Then $ G_1 \odot H$ and $ G_2 \odot H$ are $A$-cospectral.\\
$(c)$ Let $G_1$ and $G_2$ be $A$-cospectral regular graphs, $H_1$ and $H_2$ are another $A$-cospectral regular graphs. Then $G_1 \odot H_1$ and $G_2 \odot H_2$ are $A$-cospectral.
\end{cor}
\begin{exam}
Consider the two regular non isomorphic cospectral graphs $G$ and $H$ as in \cite{van2003graphs}. Graphs $C(G)$ and $C(H)$ as shown in Figure:3. Also, $G\odot K_2$ and $H\odot K_2$ are non isomorphic. If $G$ and $H$ are regular cospectral graphs, then they have the same regularity with same number of vertices and same number of edges. Since the eigenvalues of $C(G)$ \cite{jahfar2020central} depends on number of vertices, regularity and eigenvalues of $G$.  So $C(G)$ and $C(H)$ are $A$-cospectral.
 By applying Theorem 3.1 on the concerned graphs and comparing their characteristic polynomial we get  $G\odot K_2$ and $H\odot K_2$ are $A$-cospectral graphs.	 	
\end{exam}
\begin{figure}[H]
	\begin{minipage}[b]{0.5\linewidth}
		\centering
		\includegraphics[width=6.0 cm]{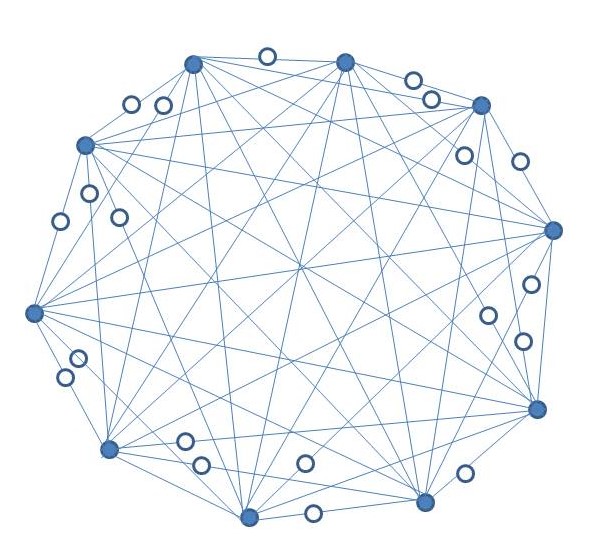}
		\caption{  $C(G)$}
		\label{pict12.jpg}
	\end{minipage}
	\hspace{0.5cm}
	\begin{minipage}[b]{0.5\linewidth}
		\centering
		\includegraphics[width=6.0 cm]{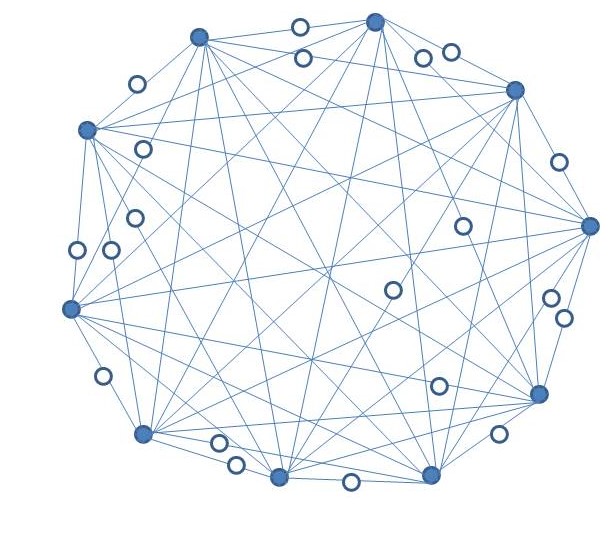}
		\caption{  $C(H)$}
		\label{pict13.jpg}
	\end{minipage}
	\caption{Figure:2 Non-regular non isomorphic cospectral graphs}
\end{figure}	
\begin{figure}[H]
	\begin{minipage}[b]{0.5\linewidth}
		\centering
		\includegraphics[width=8.0 cm]{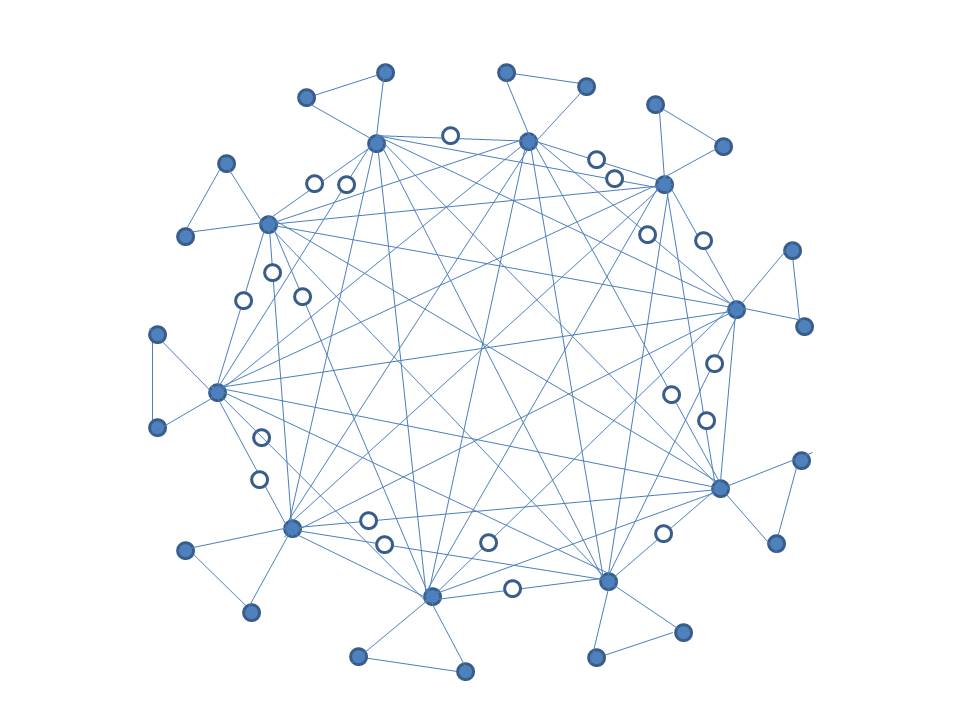}
		\caption{  $G\odot K_2$}
		\label{pict14}
	\end{minipage}
	\hspace{0.5cm}
	\begin{minipage}[b]{0.5\linewidth}
		\centering
		\includegraphics[width=8.0 cm]{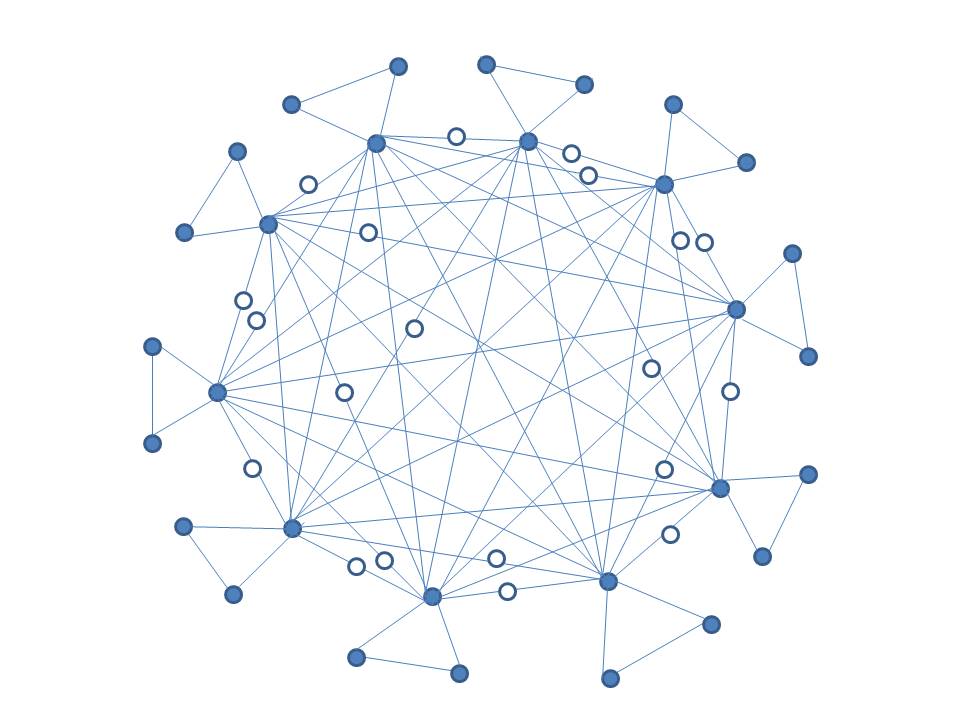}
		\caption{  $H\odot K_2$}
		\label{pict15}
	\end{minipage}
	\caption{Figure:3 Non-regular non isomorphic $A $-cospectral graphs}
\end{figure}
	Next, we consider the Laplacian characteristic polynomial of $G_1\odot G_2$ when $G_1$ is regular and $G_2$ is an arbitrary graph.
	\begin{thm}\label{th4.13}
	Let $G_1$ be an $r_1$-regular graph with $n_1$ vertices and $m_1$ edges and $G_2$ be an arbitrary graph with $n_2$ vertices. Then the Laplacian characteristic polynomial of $G_1\odot G_2$ is 
\begin{equation*}
\begin{aligned}
f(L(G_1\odot G_2),x) =& (x-2)^{m_1-n_1}\Big[x^3-x^2(n_2+r_1+3)+x(r_1+2n_1+2n_2+2)\Big]\\&\times \prod_{j=2}^{n_2}\Big(x-1-\mu_j(G_2)\Big)^{n_1}\prod_{i=2}^{n_1}\Big[x^3-x^2(n_1+n_2+3+\lambda_i(G_1)\\&+x(n_2+3n_1+2\lambda_i(G_1)+2-r_1)-n_2-2n_1+r_1-\lambda_i(G_1)) \Big].\\  
\end{aligned}
\end{equation*} 
\end{thm}
\begin{proof}
	Let $L(G_2)$ be the Laplacian matrix of $G_2$. Then by a proper labeling of vertices, the Laplacian matrix of $G_1\odot G_2$ can be written as
	\begin{align*}
	L(G_1\odot G_2)=\begin{pmatrix}
	(n_1+n_2-1)I_{n_1}-A(\bar {G_1})&-I(G_1)&-I_{n_1}\otimes J_{1\times n_2}\\
	-I(G_1)^T&2I_{m_1}& O_{m_1\times n_1n_2}\\
	-I_{n_1}\otimes J_{1\times n_2}^T&O_{ n_1n_2\times m_1}&I_{n_1}\otimes (I_{n_2}+L(G_2)) 
	\end{pmatrix}. \\
	\end{align*}
	The Laplacian characteristic polynomial of $G_1\odot G_2$ is \\
	\begin{align*}
	f(L(G_1\odot G_2),x) =&  det\left (\begin{matrix}
	(x-n_1-n_2+1)I_{n_1}+A(\bar {G_1})&I(G_1)&I_{n_1}\otimes J_{1\times n_2} \\
	I(G_1)^T&(x-2)I_{m_1}& O_{m_1\times n_1n_2}\\
	I_{n_1}\otimes J_{1\times n_2}^T&O_{ n_1n_2\times m_1}&I_{n_1}\otimes((x-1)I_{n_2}-L(G_2))  
	\end{matrix} \right).\end{align*} \\\textnormal{By Lemmas 2.1, 2.2, Definition 2.1, equation (2.0.1) and Corollary 2.6, we have}\\
	\begin{align*}
	f(L(G_1\odot G_2),x)=& det(I_{n_1}\otimes((x-1)I_{n_2}-L(G_2)))det S,\\
	\textnormal{where~S} =&\begin{pmatrix}
	(x-n_1-n_2+1)I_{n_1}+A(\bar {G_1})&{I(G_1)}\\
	I(G_1)^T&{(x-2)I_m}_1\\
	\end{pmatrix}\\&-\begin{pmatrix}
	\-I_{n_1}\otimes J_{1\times n_2} \\
	O_{{m_1}\times {n_1n_2}}
	\end{pmatrix}\left(I_{n_1}\otimes((x-1)I_{n_2}-L(G_2))\right)^{-1}\begin{pmatrix}
	\-I_{n_1}\otimes J_{1\times n_2}^T&O_{n_1n_2\times m_1}\end {pmatrix}\\
	=&\begin{pmatrix}
	(x-n_1-n_2+1)I_{n_1}+A(\bar {G_1})-\chi_{L{(G_2)}}(x-1)I_{n_1}&I(G_1)\\
	I(G_1)^T&{(x-2)I_m}_1\\
	\end{pmatrix}.
	\end{align*}Therefore,
	\begin{align*}
	det S=&det\left (\begin{matrix}
	(x-n_1-n_2+1)I_{n_1}+A(\bar {G_1})-\chi_{L{(G_2)}}(x-1)I_{n_1}&I(G_1)\\
	I(G_1)^T&{(x-2)I_m}_1
	\end{matrix} \right)\\
	= & (x-2)^{m_1} det\left((x-n_1-n_2+1)I_{n_1}+A(\bar {G_1})-\chi_{L{(G_2)}}(x-1)I_{n_1}-\frac{I(G_1)I(G_1)^T}{x-2}\right)\\ 
	= & (x-2)^{m_1} det\left((x-n_1-n_2+1)I_{n_1}+J_{n_1}-I_{n_1}-A(G_1)-\frac{n_2}{x-1}I_{n_1}-\frac{A(G_1)+r_1I_{n_1}}{x-2}\right)\\   
	= & (x-2)^{m_1} det\left(\left(x-n_1-n_2-\frac{n_2}{x-1}-\frac{r_1}{x-2}\right)I_{n_1}+J_{n_1}-(1+\frac{1}{x-2})A(G_1)\right)
	\end{align*}
	\begin{align*}
	= (x-2)^{m_1-n_1} \prod_{i=1}^{n_1}\left(\left(x-n_1-n_2-\frac{n_2}{x-1}-\frac{r_1}{x-2}\right)(x-2)+(x-2)P(\lambda_i(G_1))-(x-1)\lambda_i(G_1)\right).\\
	\end{align*}	
	Note that $\mu_1(G_2)=0$, $P(\lambda_1(G_1))=n_1$ and $P(\lambda_i(G_1))=0, i=2,...,n_1.$	
	\\Thus the characteristic polynomial of $L(G_1\odot G_2)$		 
	\begin{equation*}
	\begin{aligned}
	f(L(G_1\odot G_2),x) =& (x-2)^{m_1-n_1}\Big[x^3-x^2(n_2+r_1+3)+x(r_1+2n_2+2)\Big]\\&\times \prod_{j=2}^{n_2}\Big(x-1-\mu_j(G_2)\Big)^{n_1}\prod_{i=2}^{n_1}\Big[x^3-x^2(n_1+n_2+3+\lambda_i(G_1)\\&+x(2n_2+3n_1+2\lambda_i(G_1)+2-r_1)-2n_1+r_1-\lambda_i(G_1)) \Big].\\ 
	\end{aligned}
	\end{equation*} 
\end{proof}	
The following corollary describes the complete Laplacian spectrum of $G_1\odot G_2$, when $G_1$ is a regular graph and $G_2$ is an arbitrary graph.
\begin{cor}
	Let $G_1$ be an $r_1$-regular graph with $n_1$ vertices and $m_1$ edges and $G_2$ be an arbitrary graph with $n_2$ vertices. Then the Laplacian spectrum of $G_1\odot G_2$ consists of 
	\begin{enumerate}
		\item  $2$, repeated $m_1-n_1$ times,
		\item  $1+\mu_j(G_2) $,  repeated $n_1$ times for $ j=2,3,...,n_2$,
		\item three roots of the equation $x^3-x^2(n_2+r_1+3)+x(r_1+2n_2+2)=0,$
		\item three roots of the equation $x^3-x^2(n_1+n_2+3+\lambda_i(G_1))+x(2n_2+3n_1+2\lambda_i(G_1)+2-r_1)-2n_1+r_1-\lambda_i(G_1)) =0 $ for $ i=2,3,...,n_1.$
	\end{enumerate}
\end{cor}
	\begin{cor}
	Let $G_1$ be an $r_1$-regular graph with $n_1$ vertices and $m_1$ edges and $G_2$ be an arbitrary graph with $n_2$ vertices .
	Then the number of spanning trees of $G_1\odot G_2$ is\\ $$t(G_1\odot G_2)=\frac{2^{m_1-n_1}(r_1+2n_2+2)\prod_{j=2}^{n_2}(1+\mu_j(G_2))^{n_1}\prod_{i=2}^{n_1}(2n_1-r_1+\lambda_i(G_1))}{n_1+m_1+n_1n_2}.$$
	
\end{cor}
\begin{cor}
	Let $G_1$ be an $r_1$-regular graph with $n_1$ vertices and $m_1$ edges and $G_2$ be  an arbitrary graph with $n_2$ vertices .
	Then the Kirchhoff index of $G_1\odot G_2$ is\\\\ $Kf(G_1\odot G_2)=(n_1+m_1+n_1n_2)\bigg[\frac{m_1-n_1}{2}+\frac{n_2+r_1+3}{r_1+2n_2+2}+\displaystyle\sum_{j=2}^{n_2}\frac{n_1}{1+\mu_j(G_2)}\\~~~~~~~~~~~~~~~~~~~~~~~~~~~~~~~~~~~~~~~~~~~~~~~~~~~~~~~~~+\displaystyle\sum_{i=2}^{n_1}\frac{2n_2+3n_1+2\lambda_i(G_1)+2-r_1}{2n_1-r_1+\lambda_i(G_1)}\bigg].$
	
\end{cor}
Corollary 3.7 helps us to construct infinitely many pairs of $L$-cospectral graphs.
\begin{cor}
	$(a)$ Let $G_1$ and $G_2$ be $L$-cospectral  graphs and $H$ is an arbitarary regular graph. Then $H\odot  G_1$ and $H\odot  G_2$ are $L$-cospectral.\\
	$(b)$ 	Let $G_1$ and $G_2$ be $L$-cospectral regular graphs  and $H$ is an arbitrary  graph. Then $ G_1 \odot  H$ and $ G_2 \odot  H$ are $L$-cospectral.\\
	$(c)$ 	Let $G_1$ and $G_2$ be $L$-cospectral regular graphs, $H_1$ and $H_2$ are another $L$-cospectral regular graphs. Then $G_1 \odot  H_1$ and $G_2 \odot  H_2$ are $L$-cospectral.
\end{cor}
	Next, we consider the signless Laplacian characteristic polynomial of $G_1\odot G_2$ when $G_1$ is regular and $G_2$ is an arbitrary graph.
\begin{thm}
Let $G_1$ be an $r_1$-regular graph with $n_1$ vertices and $m_1$ edges and $G_2$ be an arbitrary graph with $n_2$ vertices.  Then the signless Laplacian characteristic polynomial of $G_1\odot G_2$ is 
	\begin{equation*}
	\begin{aligned}
	f(Q(G_1\odot G_2),x) &= (x-2)^{m_1-n_1}\prod_{j=2}^{n_2}\Big(x-1-\gamma_j(G_2)\Big)^{n_1}\\&\times\prod_{i=1}^{n_1}\Bigg[\left(x-n_1-n_2+2-\chi_{Q{(G_2)}}(x-1)-\frac{r_1}{x-2}\right)(x-2)-(x-2)P(\lambda_i(G_1))+\\&~~~~~~~~~~~~~~~~~~~~~~~~~~~~~~~~~~~~~~~~~~~~~~~~~~~~~~~~~~~~~~(x-3)\lambda_i(G_1)\Bigg]\\  
	\end{aligned}
	\end{equation*}
\end{thm}
\begin{proof}
	Let $Q(G_2)$ be the signless Laplacian matrix of $G_2$. Then by a proper labeling of vertices, the signless Laplacian matrix of $G_1\odot G_2$ can be written as
	\begin{align*}
	Q(G_1\odot G_2)=\begin{pmatrix}
	(n_1+n_2-1)I_{n_1}+A(\bar {G_1})&I(G_1)&I_{n_1}\otimes J_{1\times n_2} \\
	I(G_1)^T&2I_{m_1}& O_{m_1\times n_1n_2}\\
	I_{n_1}\otimes J_{1\times n_2}^T&O_{ n_1n_2\times m_1}&I_{n_1}\otimes (I_{n_2}+Q(G_2)) 
	\end{pmatrix}. \\
	\end{align*}
	The signless Laplacian characteristic polynomial of $G_1\odot G_2$ is \\
	\begin{align*}
	f(Q(G_1\odot G_2),x) =&  det\left (\begin{matrix}
	(x-n_1-n_2+1)I_{n_1}-A(\bar {G_1})&-I(G_1)&-I_{n_1}\otimes J_{1\times n_2} \\
	-I(G_1)^T&(x-2)I_{m_1}& O_{m_1\times n_1n_2}\\
	-I_{n_1}\otimes J_{1\times n_2}^T&O_{ n_1n_2\times m_1}&I_{n_1}\otimes((x-1)I_{n_2}-Q(G_2))  
	\end{matrix} \right).\end{align*} 	The rest of the proof is similar to that of Theorem 3.6 and hence we omit details.
\end{proof}
	Next, we consider the signless Laplacian characteristic polynomial of $G_1\odot G_2$ when $G_1$  and $G_2$ are both regular graphs.
\begin{cor}
	Let $G_i$ be an $r_i$-regular graph with $n_i$ vertices and $m_i$ edges for ${i=1,2}.$ Then the signless Laplacian characteristic polynomial of $G_1\odot G_2$ is 
	\begin{equation*}
	\begin{aligned}
	&f(Q(G_1\odot G_2),x)\\ =& (x-2)^{m_1-n_1}\Big[x^3-x^2(n_2+2n_1+2r_2+1-r_1)+x(4n_1r_2+2n_2r_2+6n_1+2n_2-4-2r_2r_1-5r_1)\\&-8n_1r_2-4n_2r_2+2r_1r_2+4r_1-4n_1+8r_2+4+6r_2r_1\Big] \prod_{j=2}^{n_2}\Big(x-1-\gamma_j(G_2)\Big)^{n_1}\\&\times\prod_{i=2}^{n_1}\Big[x^3-x^2(n_2+n_1+2r_2+1-\lambda_i(G_1))+x(2n_1r_2+2n_2r_2+3n_1+2n_2-r_1-4\\&-2r_2\lambda_i(G_1)-4\lambda_i(G_1))-4n_1r_2-4n_2r_2+2r_1r_2+3\lambda_i(G_1)+r_1+8r_2-2n_1+6r_2\lambda_i(G_1)+4 \Big].\\   
	\end{aligned}
	\end{equation*} 
\end{cor}
The following corollary illustrates the signless Laplacian spectrum of $G_1 \odot G_2$, when $G_1$ and $G_2$ are both regular graphs.
\begin{cor}
	Let $G_i$ be an $r_i$-regular graph with $n_i$ vertices and $m_i$ edges for ${i=1,2}.$ Then the signless Laplacian spectrum of $G_1\odot G_2$ consists of 
	\begin{enumerate}
		\item  $2$, repeated $m_1-n_1$ times,
		\item  $1+\gamma_j(G_2), $  repeated $n_1$ times for$~j=2,3,...,n_2,$
		\item three roots of the equation $x^3-x^2(n_2+2n_1+2r_2+1-r_1)+x(4n_1r_2+2n_2r_2+6n_1+2n_2-4-2r_2r_1-5r_1)-8n_1r_2-4n_2r_2+2r_1r_2+4r_1-4n_1+8r_2+4+6r_2r_1=0,$
		\item three roots of the equation $x^3-x^2(n_2+n_1+2r_2+1-\lambda_i(G_1))+x(2n_1r_2+2n_2r_2+3n_1+2n_2-r_1-4-2r_2\lambda_i(G_1)-4\lambda_i(G_1))-4n_1r_2-4n_2r_2+2r_1r_2+3\lambda_i(G_1)+r_1+8r_2-2n_1+6r_2\lambda_i(G_1)+4=0 $ for $ i=2,3,...,n_1.$
	\end{enumerate}
\end{cor}
The following corollary helps us to construct infinitely many pairs of $Q$-cospectral graphs.
\begin{cor}
	$(a)$	Let $G_1$ and $G_2$ be $Q$-cospectral regular graphs and $H$ is any regular graph. Then $H\odot  G_1$ and $H\odot  G_2$ are $Q$-cospectral.\\
	$(b)$ Let $G_1$ and $G_2$ be $Q$-cospectral regular graphs  and $H$ is any regular graph. Then $ G_1 \odot  H$ and $ G_2 \odot  H$ are $Q$-cospectral.\\
	$(c)$	Let $G_1$ and $G_2$ be $Q$-cospectral regular graphs, $H_1$ and $H_2$ are another $Q$-cospectral regular graphs. Then $G_1 \odot  H_1$ and $G_2 \odot  H_2$ are $Q$-cospectral.
\end{cor}
\section{Spectra of central edge corona of graphs}
In this section, we define central edge corona of two graphs and calculate their adjacency, Laplacian  and signless Laplacian spectrum.  Also, we compute the Kirchhoff index and the number of spanning trees of the resulting graphs. Using our results we establish some cospectral graphs. 
\begin{defn} 	Let $G_i$ be a graph with $n_i$ vertices and $m_i$ edges for  $i=1,2$. The central edge corona $G_1 \underline{\odot} G_2$ of $G_1$ and $G_2$ is the graph obtained by taking  $C(G_1)$ and $|\tilde{V}(G_1)|$  copies of $G_2$ and joining the $i^{th}$ vertex of $\tilde{V}(G_1)$ to every vertex in the $i^{th}$ copy of $G_2$. \\The adjacency matrix of $G_1 \underline{\odot} G_2$   can be written as
\begin{align*}
A(G_1 \underline{\odot} G_2)=\begin{pmatrix}
	A(\overline{G_1})&I(G_1)&O_{ n_1\times m_1n_2}\\
I(G_1)^T&O_{m_1 \times m_1}& I_{m_1}\otimes J_{1\times n_2}\\
O_{ m_1n_2\times n_1 } &I_{m_1}\otimes J_{1\times n_2}^T&I_{m_1}\otimes A(G_2) 
\end{pmatrix}. \\
\end{align*}
The graph $G_1  \underline{\odot} G_2$ has $n_1+m_1+m_1n_2$ vertices and $m_1+\frac{n_1(n_1-1)}{2}+m_1m_2+m_1n_2$ edges.
\begin{exam}
	Let $G_1=P_3$ and $G_2=P_2$. Then the two central edge coronas $G_1 \underline{\odot} G_2$ and $G_2 \underline{\odot} G_1$ are depicted in Figure:4. 
\end{exam}
	\begin{figure}[H]
	\begin{minipage}[b]{0.5\linewidth}
		\centering
		\includegraphics[width=9.0 cm]{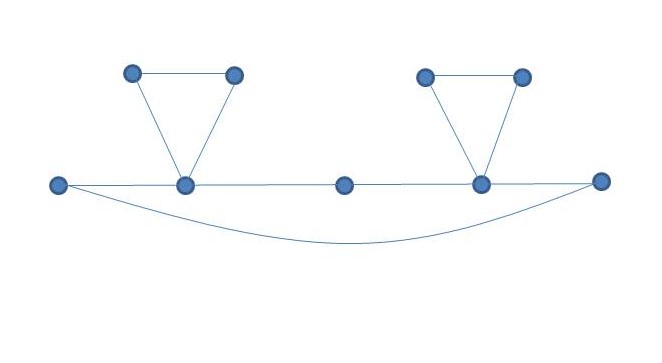}
		\caption{  $P_3 \underline{\odot} P_2$}
		\label{pict12.jpg}
	\end{minipage}
	\hspace{0.5cm}
	\begin{minipage}[b]{0.4\linewidth}
		\centering
		\includegraphics[width=7.0 cm]{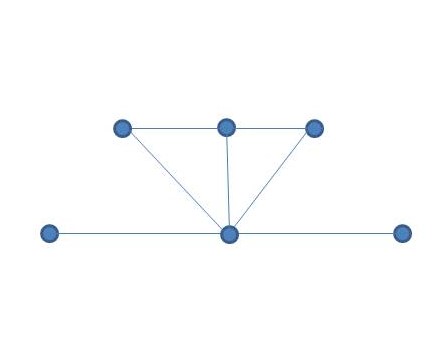}
		\caption{  $P_2 \underline{\odot}  P_3$}
		\label{pict13.jpg}
	\end{minipage}
\caption{Figure:4 An example of central edge corona graphs}
\end{figure}	
\end{defn}
First we consider the adjacency characteristic polynomial  of $G_1 \underline{\odot}  G_2$.
		\begin{thm}
		Let $G_1$ be an $r_1$-regular graph with $n_1$ vertices and $m_1$ edges and $G_2$ be an arbitrary graph with $n_2$ vertices. Then the adjacency characteristic polynomial of $G_1 \underline{\odot}  G_2$ is 
	\begin{equation*}
	\begin{aligned}
	f(A(G_1  \underline{\odot} G_2),x) =&(x-\chi_{A{(G_2)}}(x))^{m_1-n_1} \prod_{j=1}^{n_2}(x-\lambda_j(G_2))^{m_1}\prod_{i=1}^{n_1}\bigg[(x+1)(x-\chi_{A(G_2)}(x))-r_1-\\&\;\;\;\;\;\;\;\;\;\;\;\;\;\;\;\;\;\;\;\;\;\;P(\lambda_i(G_1))(x-\chi_{A(G_2)}(x))+(x-\chi_{A(G_2)}(x)-1)\lambda_i(G_1)\bigg].\\
	\end{aligned}
	\end{equation*} 
	\end{thm}
	\begin{proof}
		The characteristic polynomial of $G_1  \underline{\odot} G_2$ is \\
		\begin{align*}
		f(A(G_1  \underline{\odot} G_2),x) =&  det\left (\begin{matrix}
		xI_{n_1}-A(\bar G_1)&-I(G_1)&O_{ n_1\times m_1n_2}\\
		-I(G_1)^T&xI_{m_1}&-I_{m_1}\otimes J_{1\times n_2}\\
		O_{  m_1n_2\times n_1}&-I_{m_1}\otimes J_{1\times n_2}^T &I_{m_1}\otimes (xI_{n_2}-A(G_2)) 
		\end{matrix} \right).
		\end{align*}
		\\ \textnormal{By Lemmas 2.1, 2.2, Definition 2.1 and Corollary 2.6, we have}\\
		\begin{align*}
		f(A(G_1  \underline{\odot} G_2),x)=&\det(I_{m_1}\otimes (xI_{n_2}-A(G_2))\det S,\\
		\textnormal{where~S} =&\begin{pmatrix}
		xI_{n_1}-J_{n_1}+I_{n_1}+A( G_1)&{-I(G_1)}\\
		-I(G_1)^T&{xI_m}_1\\
		\end{pmatrix}\\\hspace{-2cm} &-\begin{pmatrix}
		O_{ n_1\times m_1n_2}\\
		-I_{m_1}\otimes J_{1\times n_2}
		\end{pmatrix}\left(I_{m_1}\otimes xI_{n_2}-A(G_2)\right)^{-1}\begin{pmatrix}
		\-O_{ m_1n_2\times n_1 }&-I_{m_1}\otimes J_{1\times n_2}^T\end {pmatrix}\\
		=&\begin{pmatrix}
		xI_{n_1}-J_{n_1}+I_{n_1}+A( G_1)&-I(G_1)\\
		-I(G_1)^T&xI_{m_1}-\chi_{A{(G_2)}}(x)I_{m_1}
		\end{pmatrix}.
		\end{align*}Therefore,\begin{align*}det~ S\\=&det\left (\begin{matrix}
		xI_{n_1}-J_{n_1}+I_{n_1}+A( G_1)&-I(G_1)\\
		-I(G_1)^T&{xI_m}_1-\chi_{A{(G_2)}}(x)I_{m_1}
		\end{matrix} \right)\\
		= & (x-\chi_{A{(G_2)}}(x))^{m_1} det\left(xI_{n_1}-J_{n_1}+I_{n_1}+A( G_1)-\frac{I(G_1)I(G_1)^T}{x-\chi_{A{(G_2)}}(x)}\right)\\ 
		= &(x-\chi_{A{(G_2)}}(x))^{m_1}  det\left((x+1)I_{n_1}-J_{n_1}+A( G_1)-\frac{A(G_1)+r_1I_{n_1}}{x-\chi_{A{(G_2)}}(x)}\right)\\ 
		= &(x-\chi_{A{(G_2)}}(x))^{m_1}  det\left((x+1-\frac{r_1}{x-\chi_{A{(G_2)}}(x)})I_{n_1}-J_{n_1}+(1-\frac{1}{x-\chi_{A{(G_2)}}(x)})A(G_1)\right)\\ 
		=&(x-\chi_{A{(G_2)}}(x))^{m_1} \prod_{i=1}^{n_1}\left((x+1-\frac{r_1}{x-\chi_{A{(G_2)}}(x)})-P(\lambda_i(G_1))+(1-\frac{1}{x-\chi_{A{(G_2)}}(x)})\lambda_i(G_1)\right)\\
	    =&(x-\chi_{A{(G_2)}}(x))^{m_1-n_1} \prod_{i=1}^{n_1}\bigg[(x+1)(x-\chi_{A(G_2)}(x))-r_1-\\&\;\;\;\;\;\;\;\;\;\;\;\;\;\;\;\;\;\;\;\;\;\;\;\;\;\;\;\;\;\;\;\;\;P(\lambda_i(G_1))(x-\chi_{A(G_2)}(x))+(x-\chi_{A(G_2)}(x)-1)\lambda_i(G_1)\bigg].\\
		\end{align*}
		Therefore the characteristic polynomial of $G_1  \underline{\odot} G_2$ is		 
		\begin{equation*}
		\begin{aligned}
		f(A(G_1  \underline{\odot} G_2),x) =&(x-\chi_{A{(G_2)}}(x))^{m_1-n_1} \prod_{j=1}^{n_2}(x-\lambda_j(G_2))^{m_1}\prod_{i=1}^{n_1}\bigg[(x+1)(x-\chi_{A(G_2)}(x))-r_1-\\&\;\;\;\;\;\;\;\;\;\;\;\;\;P(\lambda_i(G_1))(x-\chi_{A(G_2)}(x))+(x-\chi_{A(G_2)}(x)-1)\lambda_i(G_1)\bigg].\\
		\end{aligned}
		\end{equation*} 
	\end{proof}	
 The following corollary  gives adjacency characteristic polynomial of $G_1 \underline{\odot} G_2$  when $G_1$ and $G_2$ are both regular graphs.
\begin{cor}
	Let $G_i$ be an $r_i$-regular graph with $n_i$ vertices and $m_i$ edges for ${i=1,2}.$ Then the adjacency characteristic polynomial of $G_1 \underline{\odot} G_2$ is
	\begin{equation*}
	\begin{aligned}
	f(A(G_1 \underline{\odot} G_2),x) =&(x^2-xr_2-n_2)^{m_1-n_1}\bigg[x^3-x^2(r_2-1+n_1-r_1)+x(-n_2-r_2+n_1r_2\\&-r_1r_2-2r_1)-r_1n_2+n_1n_2-n_2+2r_1r_2\bigg]\prod_{j=2}^{n_2}(x-\lambda_j(G_2))^{m_1}\\& \prod_{i=2}^{n_1}\bigg[x^3-x^2(r_2-\lambda_i(G_1)-1)-x(n_2+r_2+r_2\lambda_i(G_1)+r_1+\lambda_i(G_1))\\&~~~~~~~~~~~~~~~~~~~~~~~~~~~~~~~~~~~~~~~-n_2\lambda_i(G_1)+r_2\lambda_i(G_1)+r_1r_2-n_2\bigg].\\ 
	\end{aligned}
	\end{equation*} 	
\end{cor} 
The following corollary gives the spectrum of $G_1 \underline{\odot} G_2$, when $G_1$ and $G_2$ are both regular graphs.\\\\ 
\begin{cor}
Let $G_i$ be an $r_i$-regular graph with $n_i$ vertices and $m_i$ edges for ${i=1,2}.$ Then the adjacency spectrum of $G_1  \underline{\odot}  G_2$ consists of 
\begin{enumerate}
	\item $\lambda_j(G_2),$ repeated $m_1$ times for $j=2,3,...,n_2$,
	\item two roots of the equation $x^2-xr_2-n_2=0,$ each root repeated $m_1-n_1$ times,
	\item three roots of the equation $$x^3-x^2(r_2-1+n_1-r_1)+x(-n_2-r_2+n_1r_2-r_1r_2-2r_1)-r_1n_2+n_1n_2-n_2+2r_1r_2=0,$$
	\item three roots of the equation $x^3-x^2(r_2-\lambda_i(G_1)-1)-x(n_2+r_2+r_2\lambda_i(G_1)+r_1+\lambda_i(G_1))-n_2\lambda_i(G_1)+r_2\lambda_i(G_1)+r_1r_2-n_2=0 $ for $ i=2,3,...,n_1.$
\end{enumerate}
\end{cor}
\begin{exam}
	Let $G_1=K_{3,3}$ and $G_2=K_2$. Then the adjacency  eigenvalues of $G_1$ are 0 (multiplicity $4$) and $\pm 3$, eigenvalues of $G_2$ are $\pm 1$. Therefore, the adjacency  eigenvalues of $G_1  \underline{\odot}  G_2$ are $-1$ (multiplicity $9$),  roots of the equation $x^2-x-2=0$ (each root with multiplicity $3$), roots of the equation $x^3-3x^2-6x+10=0$, roots of the equation $x^3-3x^2+4=0$ and roots of the equation $x^3-6x+1=0$ (each root with multiplicity $4$).\\
\end{exam}
 Next, we shall consider the adjacency spectrum of $G_1  \underline{\odot}  G_2$ when $G_1$ is a regular graph and $G_2=K_{p,q}$,  $G_2$ is non-regular if $p\ne q$.
\begin{cor}
	Let $G_1$ be an $r_1$-regular graph with $n_1$ vertices and $m_1$ edges . Then the adjacency spectrum of  $ G_1\underline{\odot} K_{p,q}$ consists of
	\begin{enumerate}
		\item  $0$, repeated  $m_1(p+q-2)$ times,
		\item three roots of the equation $x^3-x(pq+p+q)-2pq=0$, each root repeated $m_1-n_1$ times,
		\item four roots of the equation $x^4+(1+\lambda_i)x^3-(pq+p+q+r_1+\lambda_i(G_1))x^2+(-3pq-p-q-p\lambda_i(G_1)-q\lambda_i(G_1)-pq\lambda_i(G_1))x-2pq-\lambda_i(G_1)pq+r_1pq=0$ for $ i=2,...,n_1,$
		\item four roots of the equation_ $x^4+(1+r_1-n_1)x^3-(pq+p+q+2r_1)x^2+(-3pq-p-q+pqn_1+pn_1+qn_1-pqr_1-pr_1-qr_1)x-2pq+2pqn_1=0.$
	\end{enumerate}
\end{cor} 
Corollary 4.3\label{key} helps us to construct infinitely many pairs of $A$-cospectral graphs.
\begin{cor}
	$(a)$ Let $G_1$ and $G_2$ be $A$-cospectral regular graphs and $H$ is any regular graph. Then $ H \underline{\odot}  G_1$ and $ H \underline{\odot}  G_2$ are $A$-cospectral.\\
	$(b)$
	Let $G_1$ and $G_2$ be $A$-cospectral regular graphs and $H$ is any regular graph. Then $ G_1  \underline{\odot}  H$ and $ G_2  \underline{\odot}  H$ are $A$-cospectral.\\
	$(c)$ Let $G_1$ and $G_2$ be A-cospectral regular graphs, $H_1$ and $H_2$ are another $A$-cospectral regular graphs. Then $G_1  \underline{\odot}  H_1$ and $G_2  \underline{\odot}  H_2$ are $A$-cospectral.
\end{cor}
\begin{exam}
	Consider the two regular non isomorphic cospectral graphs $G$ and $H$ as in \cite{van2003graphs}. By similar arguments as in Example 3.3, we have $G  \underline{\odot}  K_2$ and $H  \underline{\odot}  K_2$ are $A$-cospectral. 
\end{exam}
%
\begin{figure}[H]
	\begin{minipage}[b]{0.5\linewidth}
		\centering
		\includegraphics[width=8.0 cm]{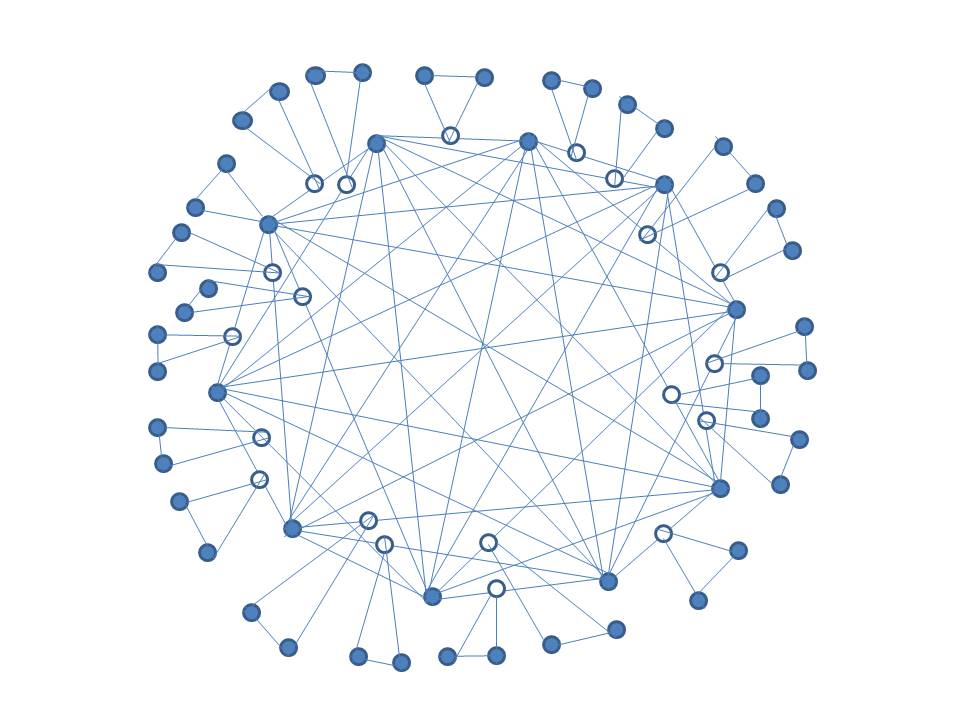}
		\caption{  $G\underline{\odot}  K_2$}
		\label{pict14}
	\end{minipage}
	\hspace{0.5cm}
	\begin{minipage}[b]{0.5\linewidth}
		\centering
		\includegraphics[width=8.0 cm]{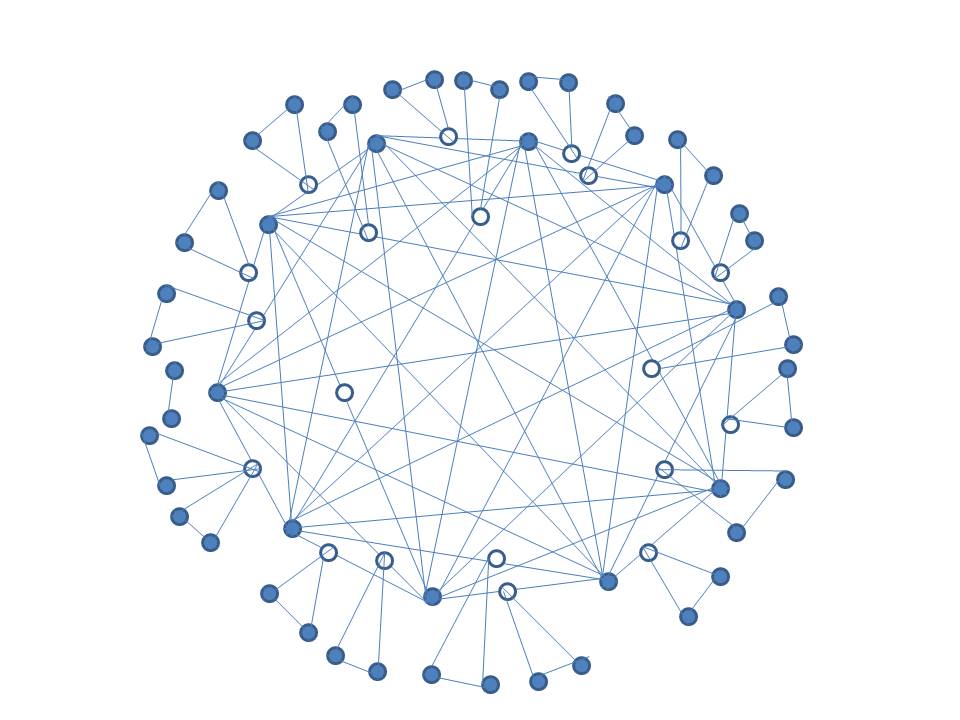}
		\caption{  $H \underline{\odot}   K_2$}
		\label{pict15}
	\end{minipage}
	\caption{Figure:5 Non-regular non isomorphic $A $-cospectral graphs}
\end{figure}
	Next, we discuss the Laplacian characteristic polynomial of $G_1\underline{\odot}  G_2$, when $G_1$ is a regular graph and $G_2$ is an arbitrary graph.	
	\begin{thm}\label{th4.13}
	Let $G_1$ be an $r_1$-regular graph with $n_1$ vertices and $m_1$ edges and $G_2$ be an arbitrary graph with $n_2$ vertices. Then the Laplacian characteristic polynomial of $G_1 \underline{\odot}  G_2$ is 
\begin{equation*}
\begin{aligned}
f(L(G_1 \underline{\odot}  G_2),x) =& (x^2-x(n_2+3)+2)^{m_1-n_1}\Big[x^3-x^2(n_2+r_1+3)+x(3r_1+n_2r_1+2)\Big]\\& \times \prod_{j=2}^{n_2}\Big(x-1-\mu_j(G_2)\Big)^{m_1}\prod_{i=2}^{n_1}\Big[x^3-x^2(n_2+n_1+3+\lambda_i(G_1))\\&+x(2+n_1n_2+3n_1-r_1+n_2\lambda_i(G_1)+4\lambda_i(G_1))-2n_1+r_1-\lambda_i(G_1) \Big].\\   
\end{aligned}
\end{equation*}
\end{thm}
\begin{proof}
	Let $L(G_2)$ be the Laplacian matrix of $G_2$. Then by a proper labeling of vertices, the Laplacian matrix of $G_1 \underline{\odot}  G_2$ can be written as
	\begin{align*}
	L(G_1 \underline{\odot}  G_2)=\begin{pmatrix}
	(n_1-1)I_{n_1}-A(\bar {G_1})&-I(G_1)&O_{ n_1\times m_1n_2}\\
	-I(G_1)^T&(n_2+2)I_{m_1}& -I_{m_1}\otimes J_{1\times n_2}\\
	O_{ m_1n_2\times n_1 }&-I_{m_1}\otimes J_{1\times n_2}^T&I_{m_1}\otimes (I_{n_2}+L(G_2)) 
	\end{pmatrix}. \\
	\end{align*}
	The Laplacian characteristic polynomial of $G_1 \underline{\odot}  G_2$ is \\
	\begin{align*}
	f(L(G_1 \underline{\odot}  G_2),x) =&  det\left (\begin{matrix}
	(x-n_1+1)I_{n_1}+A(\bar {G_1})&I(G_1)&O_{ n_1\times m_1n_2}\\
	I(G_1)^T&(x-n_2-2)I_{m_1}& I_{m_1}\otimes J_{1\times n_2}\\
	O_{ m_1n_2\times n_1 }&I_{m_1}\otimes J_{1\times n_2}^T&I_{m_1}\otimes((x-1)I_{n_2}-L(G_2))  
	\end{matrix} \right).\end{align*} \\\textnormal{By Lemmas 2.1, 2.2, Corollary 2.6, equation (2.0.1) and Definition 2.1, we have}\\
	\begin{align*}
	f(L(G_1 \underline{\odot}  G_2),x)=& det(I_{m_1}\otimes((x-1)I_{n_2}-L(G_2)))det S,\\
	\textnormal{where~S} =&\begin{pmatrix}
	(x-n_1+1)I_{n_1}+A(\bar {G_1})&{I(G_1)}\\
	I(G_1)^T&{(x-n_2-2)I_m}_1\\
	\end{pmatrix}\\&-\begin{pmatrix}
	O_{ n_1\times m_1n_2}\\
	I_{m_1}\otimes J_{1\times n_2}
	\end{pmatrix}\left(I_{m_1}\otimes((x-1)I_{n_2}-L(G_2))\right)^{-1}\begin{pmatrix}
	O_{ m_1n_2\times n_1 }&I_{m_1}\otimes J_{1\times n_2}^T\end {pmatrix}\\
	=&\begin{pmatrix}
	(x-n_1+1)I_{n_1}+A(\bar {G_1})&I(G_1)\\
	I(G_1)^T&{(x-n_2-2)I_m}_1-\chi_{L{(G_2)}}(x-1)I_{m_1}\\
	\end{pmatrix}.
	\end{align*}Therefore,\\ det S 
	\begin{align*}
	=&det\left (\begin{matrix}
	(x-n_1+1)I_{n_1}+A(\bar {G_1})&I(G_1)\\
	I(G_1)^T&{(x-n_2-2)I_m}_1-\chi_{L{(G_2)}}(x-1)I_{m_1}
	\end{matrix} \right)\\
	= & (x-n_2-2-\chi_{L{(G_2)}}(x-1))^{m_1} det\left((x-n_1+1)I_{n_1}+A(\bar {G_1})-\frac{I(G_1)I(G_1)^T}{(x-n_2-2-\chi_{L{(G_2)}}(x-1))}\right)\\ 
	= & (x-n_2-2-\chi_{L{(G_2)}}(x-1))^{m_1} \\&\times det\left((x-n_1+1)I_{n_1}+J_{n_1}-I_{n_1}-A(G_1)-\frac{A(G_1)+r_1I_{n_1}}{(x-n_2-2-\chi_{L{(G_2)}}(x-1))}\right)
	\end{align*} 
	\begin{align*}
	= & (x-n_2-2-\chi_{L{(G_2)}}(x-1))^{m_1} \\& det\left(\left(x-n_1-\frac{r_1}{(x-n_2-2-\chi_{L{(G_2)}}(x-1))}\right)I_{n_1}+J_{n_1}-(1+\frac{1}{(x-n_2-2-\chi_{L{(G_2)}}(x-1))})A(G_1)\right)\\ 
	= & (x^2-x(n_2+3)+2)^{m_1-n_1} \prod_{i=1}^{n_1}\bigg[\left((x-n_1)(x^2-x(n_2+3)+2)-r_1(x-1)\right)+\\&\;\;\;\;\;\;\;\;\;\;\;\;\;\;\;\;\;\;\;\;\;\;\;\;\;\;\;\;\;\;\;\;\;\;\;(x^2-x(n_2+3)+2)P(\lambda_i(G_1))-(x^2-x(n_2+3)+3)\lambda_i(G_1)\bigg].
	\end{align*}
	Note that $\mu_1(G_2)=0$, $P(\lambda_1(G_1))=n_1$, $P(\lambda_i(G_1))=0, i=2,...,n_1$ .\\	
	Thus the characteristic polynomial of $L(G_1 \underline{\odot}  G_2)$ is		 
	\begin{equation*}
	\begin{aligned}
	f(L(G_1 \underline{\odot}  G_2),x) =& (x^2-x(n_2+3)+2)^{m_1-n_1}\Big[x^3-x^2(n_2+r_1+3)+x(3r_1+n_2r_1+2)\Big]\\& \times \prod_{j=2}^{n_2}\Big(x-1-\mu_j(G_2)\Big)^{m_1}\prod_{i=2}^{n_1}\Big[x^3-x^2(n_2+n_1+3+\lambda_i(G_1))\\&+x(2+n_1n_2+3n_1-r_1+n_2\lambda_i(G_1)+4\lambda_i(G_1))-2n_1+r_1-\lambda_i(G_1) \Big].\\  
	\end{aligned}
	\end{equation*} 
\end{proof}	
The following corollary illustrates the complete Laplacian spectrum of $G_1 \underline{\odot}  G_2$, when $G_1$ is a regular graph and $G_2$ is an arbitrary graph.
\begin{cor}
	Let $G_1$ be an $r_1$-regular graph with $n_1$ vertices and $m_1$ edges and $G_2$ be an arbitrary graph with $n_2$ vertices. Then the Laplacian spectrum of $G_1 \underline{\odot}  G_2$ consists of 
	\begin{enumerate}
		\item  two roots of the equation $x^2-x(n_2+3)+2=0$, each root repeated $m_1-n_1$ times,
		\item  $1+\mu_j(G_2)$,  repeated $m_1$ times for $j=2,3,...,n_2$,
		\item three roots of the equation $x^3-x^2(n_2+r_1+3)+x(3r_1+n_2r_1+2)=0,$
		\item three roots of the equation $x^3-x^2(n_2+n_1+3+\lambda_i(G_1))+x(2+n_1n_2+3n_1-r_1+n_2\lambda_i(G_1)+4\lambda_i(G_1))-2n_1+r_1-\lambda_i(G_1) =0 $ for $ i=2,3,...,n_1.$
	\end{enumerate}
\end{cor}

\begin{cor}
	Let $G_1$ be an $r_1$-regular graph with $n_1$ vertices and $m_1$ edges and $G_2$ be an arbitrary graph with $n_2$ vertices.
	Then the number of spanning trees of $G_1 \underline{\odot}  G_2$ is\\\\ $$t(G_1 \underline{\odot}  G_2)=\frac{2^{m_1-n_1}(3r_1+n_2r_1+2)\prod_{j=2}^{n_2}(1+\mu_j(G_2))^{m_1}\prod_{i=2}^{n_1}(2n_1-r_1+\lambda_i(G_1)}{n_1+m_1+m_1n_2}.$$
	
\end{cor}
\begin{cor}
	Let $G_1$ be an $r_1$-regular graph with $n_1$ vertices and $m_1$ edges and $G_2$ be  an arbitrary graph with $n_2$ vertices.
	Then the Kirchhoff index of $G_1 \underline{\odot}  G_2$ is\\\\ $Kf(G_1 \underline{\odot}  G_2)=(n_1+m_1+m_1n_2)\bigg[\frac{(m_1-n_1)(n_2+3)}{2}+\frac{n_2+r_1+3}{3r_1+n_2r_1+2}+\displaystyle\sum_{j=2}^{n_2}\frac{m_1}{1+\mu_j(G_2)}+\\~~~~~~~~~~~~~~~~~~~~~~~~~~~~~~~~~~~~~~~~~~~~~~~~~~~~~~~~\displaystyle\sum_{i=2}^{n_1}\frac{2n_1-r_1+\lambda_i(G_1)}{2+n_1n_2+3n_1-r_1+n_2\lambda_i(G_1)+4\lambda_i(G_1)}\bigg].$
	
\end{cor}
The following corollary help us to construct infinitely many pairs of $L$-cospectral graphs.
\begin{cor}
	$(a)$ Let $G_1$ and $G_2$ be $L$-cospectral regular graphs and $H$ is an arbitarary graph. Then $H \underline{\odot}   G_1$ and $H \underline{\odot}   G_2$ are $L$-cospectral.\\
	$(b)$ 	Let $G_1$ and $G_2$ be $L$-cospectral regular graphs  and $H$ is an arbitrary  graph. Then $ G_1  \underline{\odot}   H$ and $ G_2  \underline{\odot}   H$ are $L$-cospectral.\\
	$(c)$ 	Let $G_1$ and $G_2$ be L-cospectral regular graphs, $H_1$ and $H_2$ are another $L$-cospectral regular graphs. Then $G_1  \underline{\odot}   H_1$ and $G_2  \underline{\odot}   H_2$ are $L$-cospectral.
\end{cor}
\begin{thm}
	Let $G_1$ be an $r_1$-regular graph with $n_1$ vertices and $m_1$ edges and $G_2$ be an arbitrary graph with $n_2$ vertices.  Then the signless Laplacian characteristic polynomial of $G_1\underline{\odot} G_2$ is 
	\begin{equation*}
	\begin{aligned}
	f(Q(G_1\underline{\odot} G_2),x) =& \left(x-n_2-2-\chi_{Q{(G_2)}}(x-1)\right)^{m_1-n_1}\prod_{j=2}^{n_2}\Big(x-1-\gamma_j(G_2)\Big)^{m_1}\\& \times\prod_{i=2}^{n_1}\Big[\Big((x-n_1+2)(x-n_2-2-\chi_{Q{(G_2)}}(x-1))-r_1\Big)I_{n_1}\\&-(x-n_2-2-\chi_{Q{(G_2)}}(x-1))P(\lambda_i(G_1))+\Big(x-n_2-3-\chi_{Q{(G_2)}}(x-1)\Big)\lambda_i(G_1) \Big].\\  
	\end{aligned}
	\end{equation*} 
\end{thm}
\begin{proof}
	Let $Q(G_2)$ be the signless Laplacian matrix of $G_2$. Then by a proper labeling of vertices, the Laplacian matrix of $G_1\underline{\odot} G_2$ can be written as
	\begin{align*}
	Q(G_1\underline{\odot} G_2)=\begin{pmatrix}
	(n_1-1)I_{n_1}+A(\bar {G_1})&I(G_1)&O_{ n_1\times m_1n_2}\\
	I(G_1)^T&(n_2+2)I_{m_1}& I_{m_1}\otimes J_{1\times n_2}\\
	O_{ m_1n_2\times n_1 }&I_{m_1}\otimes J_{1\times n_2}^T&I_{m_1}\otimes (I_{n_2}+Q(G_2)) 
	\end{pmatrix}. \\
	\end{align*}
	The signless Laplacian characteristic polynomial of $G_1\underline{\odot} G_2$ is \\
	\begin{align*}
	f(Q(G_1\underline{\odot} G_2),x) =&  det\left (\begin{matrix}
	(x-n_1+1)I_{n_1}-A(\bar {G_1})&-I(G_1)&O_{ n_1\times m_1n_2}\\
	-I(G_1)^T&(x-n_2-2)I_{m_1}& -I_{m_1}\otimes J_{1\times n_2}\\
	O_{ m_1n_2\times n_1 }&-I_{m_1}\otimes J_{1\times n_2}^T&I_{m_1}\otimes((x-1)I_{n_2}-Q(G_2))  
	\end{matrix} \right).\end{align*} 
	The rest of the proof is similar to that of Theorem 4.6 and hence we omit details.
\end{proof}
Next, we consider the signless Laplacian characteristic polynomial of $G_1\underline{\odot} G_2$ when $G_1$  and $G_2$ are both regular graphs.
	\begin{cor}
		Let $G_i$ be an $r_i$-regular graph with $n_i$ vertices and $m_i$ edges for ${i=1,2}.$ Then the signless Laplacian characteristic polynomial of $G_1\underline{\odot} G_2$ is 
	\begin{equation*}
	\begin{aligned}
	f(Q(G_1\underline{\odot} G_2),x) =& (x^2-x(2r_2+n_2+3)+2n_2r_2+4r_2+2)^{m_1-n_1}\Big[x^3-x^2(n_2+2r_2+2n_1+1-r_1)+\\&x(n_1n_2+4r_2n_1+2n_2r_2+6n_1-2n_2-4-n_2r_1-2r_2-5r_1+n_1n_2)\\&-4n_1n_2r_2-8n_1r_2+4n_2r_2+8r_1r_2-4n_1+4r_1+8r_2+4+2n_2r_2r_1-2n_1n_2 \Big]\\& \times\prod_{j=2}^{n_2}\Big(x-1-\gamma_j(G_2)\Big)^{m_1}\prod_{i=2}^{n_1}\Big[x^3-x^2(n_2+2r_2+n_1+1-\lambda_i(G_1))\\&+x(n_1n_2+2r_2n_1+2n_2r_2+3n_1-2n_2-r_1-4-n_2\lambda_i(G_1)-2r_2-4\lambda_i(G_1))\\&-2n_1n_2r_2-4n_1r_2+4n_2r_2+2r_1r_2-2n_1+r_1+8r_2+4+2n_2r_2\lambda_i(G_1)\\&\;\;\;\;\;\;\;\;\;\;\;\;\;\;\;\;\;\;\;\;\;\;\;\;\;\;\;\;\;\;\;\;\;\;\;\;\;\;\;\;\;\;\;\;\;\;\;\;\;\;\;\;\;\;\;\;\;\;\;\;\;\;\;\;\;\;\;\;+6r_2\lambda_i(G_1)+3\lambda_i(G_1)) \Big].\\  
	\end{aligned}
	\end{equation*} 
\end{cor}
	Next, we consider the signless Laplacian spectra of  $G_1\underline{\odot} G_2$, when $G_1$  and $G_2$ are both regular graphs.\\
\begin{cor}
	Let $G_i$ be an $r_i$-regular graph with $n_i$ vertices and $m_i$ edges for ${i=1,2}.$ Then the signless Laplacian spectrum of $G_1\underline{\odot} G_2$ consists of 
	\begin{enumerate}
		\item  two roots of the equation $x^2-x(2r_2+n_2+3)+2n_2r_2+4r_2+2=0$, each root  repeated  $m_1-n_1$ times,
		\item  $1+\gamma_j(G_2) $,  repeated $m_1$ times for $  j=2,3,...,n_2,$
		\item three roots of the equation $x^3-x^2(n_2+2r_2+2n_1+1-r_1)+x(n_1n_2+4r_2n_1+2n_2r_2+6n_1-2n_2-4-n_2r_1-2r_2-5r_1+n_1n_2)-4n_1n_2r_2-8n_1r_2+4n_2r_2+8r_1r_2-4n_1+4r_1+8r_2+4+2n_2r_2r_1-2n_1n_2 =0,$
		\item three roots of the equation $x^3-x^2(n_2+2r_2+n_1+1-\lambda_i(G_1))+\\x(n_1n_2+2r_2n_1+2n_2r_2+3n_1-2n_2-r_1-4-n_2\lambda_i(G_1)-2r_2-4\lambda_i(G_1))-2n_1n_2r_2-4n_1r_2+4n_2r_2+2r_1r_2-2n_1+r_1+8r_2+4+2n_2r_2\lambda_i(G_1)+6r_2\lambda_i(G_1)+3\lambda_i(G_1)) =0$ for $i=2,...,n_1.$
	\end{enumerate}
\end{cor}
Corollary 4.12 enables us to construct infinitely many pairs of $Q$-cospectral graphs.
\begin{cor}
	$(a)$	Let $G_1$ and $G_2$ be $Q$-cospectral regular graphs and $H$ is any regular graph. Then $H\underline{\odot}  G_1$ and $H\underline{\odot}  G_2$ are $Q$-cospectral.\\
	$(b)$ Let $G_1$ and $G_2$ be $Q$-cospectral regular graphs  and $H$ is any regular graph. Then $ G_1 \underline{\odot}  H$ and $ G_2 \underline{\odot}  H$ are $Q$-cospectral.\\
	$(c)$	Let $G_1$ and $G_2$ be Q-cospectral regular graphs, $H_1$ and $H_2$ are another $Q$-cospectral regular graphs. Then $G_1 \underline{\odot}  H_1$ and $G_2 \underline{\odot}  H_2$ are $Q$-cospectral.
\end{cor}
\section{Spectra of central edge neighborhood corona of graphs}
 In this section, we define a new graph operation called central edge neighborhood corona of graphs and determine its adjacency spectrum, Laplacian spectrum, and signless Laplacian spectrum. Also, our results leads us to construct new pairs of cospctral graphs. It is possible to calculate the number of spanning trees and the Kirchhoff index of the resulting graphs.   
 \begin{op} Let $G_i$ be a graph with $n_i$ vertices and $m_i$ edges for  $i=1,2$. Then the central edge neighborhood corona $G_1 \underline{\boxdot} G_2$ of two graphs $G_1$ and $G_2$ is the graph obtained by taking $C(G_1)$ and $|\tilde{V}(G_1)|$  copies of $G_2$ and joining the neighbors of the $i^{th}$ vertex of $\tilde{V}(G_1)$  to every vertex in the $i^{th}$ copy of $G_2$. \\The adjacency matrix of $G_1 \underline{\boxdot} G_2$ can be written as
 	\begin{align*}
 	A(G_1 \underline{\boxdot} G_2)=\begin{pmatrix}
 	A(\overline{G_1})&I(G_1)&I(G_1) \otimes J_{1\times n_2}\\
 	I(G_1)^T&O_{m_1\times m_1}& O_{m_1\times n_1n_2} \\
 	I(G_1)^T \otimes J_{1\times n_2}^T &O_{ n_1n_2\times m_1}&I_{m_1}\otimes A(G_2) 
 	\end{pmatrix}. \\
 	\end{align*}
  The graph $G_1  \underline{\boxdot} G_2$ has $n_1+m_1+m_1n_2$ vertices and $m_1+\frac{n_1(n_1-1)}{2}+m_1m_2+2m_1n_2$ edges .\\
  \begin{exam}
  	Let $G_1=P_3$ and $G_2=P_2$. Then the two central edge coronas $G_1 \underline{\boxdot} G_2$ and $G_2 \underline{\boxdot} G_1$ are depicted in Figure:6.
  \end{exam}
 		\begin{figure}[H]
 		\begin{minipage}[b]{0.5\linewidth}
 			\centering
 			\includegraphics[width=9.0 cm]{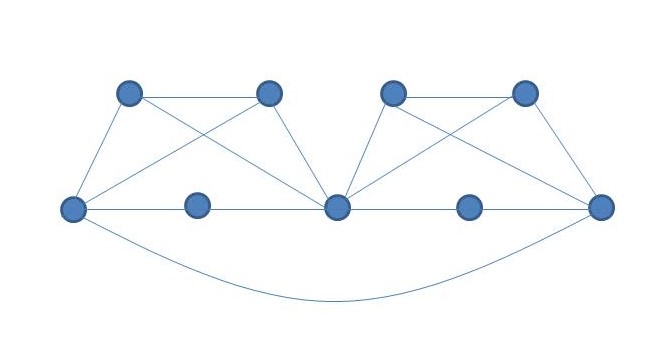}
 			\caption{  $P_3 \underline{\boxdot} P_2$}
 			\label{pict12.jpg}
 		\end{minipage}
 		\hspace{0.5cm}
 		\begin{minipage}[b]{0.4\linewidth}
 			\centering
 			\includegraphics[width=9.0 cm]{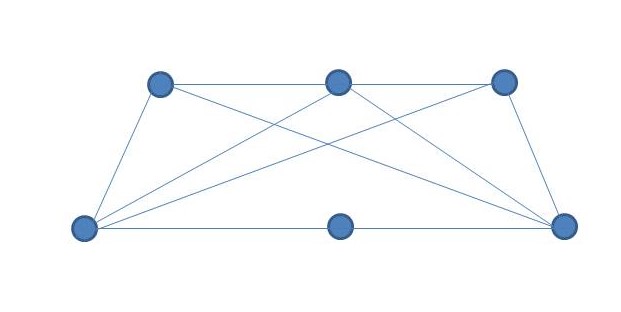}
 			\caption{  $P_2 \underline{\boxdot} P_3$}
 			\label{pict13.jpg}
 		\end{minipage}
 	\caption{Figure:6 An example of central edge neighborhood corona graphs}
 	\end{figure}	
 \end{op}
First we calculate the adjacency characteristic polynomial of $G_1 \underline{\boxdot}G_2$.
	\begin{thm}
	Let $G_1$ be an $r_1$-regular graph with $n_1$ vertices and $m_1$ edges and $G_2$ be an arbitrary graph with $n_2$ vertices. Then the adjacency characteristic polynomial of  $G_1 \underline{\boxdot}G_2$ is
	\begin{equation*}
		\begin{aligned}
			f(A(G_1 \underline{\boxdot}G_2),x)=&  x^{m_1}\prod_{j=1}^{n_2}(x-\lambda_j(G_2))^{m_1}\\& \times \prod_{i=1}^{n_1}\left((x+1-r_1\chi_{A(G_2)}(x)-\frac{r_1}{x})-P(\lambda_i(G_1))+\left(1-\chi_{A{(G_2)}}(x)-\frac{1}{x}\right)\lambda_i(G_1)\right).\\  
		\end{aligned}
	\end{equation*}
\end{thm}
\begin{proof}
	The characteristic polynomial of $G_1 \underline{\boxdot}G_2$ is \\
	\begin{align*}
	f(A(G_1 \underline{\boxdot}G_2),x) =&  det\left (\begin{matrix}
	xI_{n_1}-A(\bar G_1)&-I(G_1)&-I(G_1) \otimes J_{1\times n_2}\\
	-I(G_1)^T&xI_{m_1}& O_{m_1\times n_1n_2}\\
	-I(G_1)^T \otimes J_{1\times n_2}^T &O_{ n_1n_2\times m_1}&I_{m_1}\otimes (xI_{n_2}-A(G_2)) 
	\end{matrix} \right).
	\end{align*}
	\\ \textnormal{By Lemmas 2.1, 2.2, Definition 2.1 and Corollary 2.6, we have}\\
	\begin{align*}
	f(A(G_1 \underline{\boxdot}G_2),x)=&\det(I_{m_1}\otimes (xI_{n_2}-A(G_2))\det S,\\
	\textnormal{where~S} =&\begin{pmatrix}
	xI_{n_1}-J_{n_1}+I_{n_1}+A( G_1)&{-I(G_1)}\\
	-I(G_1)^T&{xI_m}_1\\
	\end{pmatrix}\\\hspace{-2cm} &-\begin{pmatrix}
	- I(G_1)\otimes J_{1\times n_2}\\
	O_{m_1\times n_1n_2}
	\end{pmatrix}\left(I_{m_1}\otimes xI_{n_2}-A(G_2)\right)^{-1}\begin{pmatrix}
	-{I(G_1)^T \otimes J_{1\times n_2}^T}&O_{ n_1n_2\times m_1}\end {pmatrix}\\
	=&\begin{pmatrix}
	xI_{n_1}-J_{n_1}+I_{n_1}+A( G_1)-\chi_{A{(G_2)}}(x)I(G_1)I(G_1)^T&-I(G_1)\\
	-I(G_1)^T&xI_{m_1}
	\end{pmatrix}.\\
	\det S=&det\left (\begin{matrix}
	xI_{n_1}-J_{n_1}+I_{n_1}+A( G_1)-\chi_{A{(G_2)}}(x)I(G_1)I(G_1)^T&-I(G_1)\\
	-I(G_1)^T&{xI_m}_1
	\end{matrix} \right)\\
	= & x^{m_1} det\left(xI_{n_1}-J_{n_1}+I_{n_1}+A( G_1)-\chi_{A{(G_2)}}(x)I(G_1)I(G_1)^T-\frac{I(G_1)I(G_1)^T}{x}\right)\\ 
	= & x^{m_1} det\left(xI_{n_1}-J_{n_1}+I_{n_1}+A( G_1)-\chi_{A{(G_2)}}(x)(A(G_1)+r_1I_{n_1})-\frac{A(G_1)+r_1I_{n_1}}{x}\right)\\ 
	= & x^{m_1} det\left(\left(x+1-r_1\chi_{A(G_2)}(x)-\frac{r_1}{x}\right)I_{n_1}-J_{n_1}+\left(1-\chi_{A{(G_2)}}(x)-\frac{1}{x}\right)A(G_1)\right).\\ 
,	\end{align*}
	Therefore the characteristic polynomial of $G_1 \underline{\boxdot}G_2$ is
	\begin{equation*}
	\begin{aligned}
	f(A(G_1 \underline{\boxdot}G_2),x)=&  x^{m_1}\prod_{j=1}^{n_2}(x-\lambda_j(G_2))^{m_1}\\& \times \prod_{i=1}^{n_1}\left(\left(x+1-r_1\chi_{A(G_2)}(x)-\frac{r_1}{x}\right)-P(\lambda_i(G_1))+\left(1-\chi_{A{(G_2)}}(x)-\frac{1}{x}\right)\lambda_i(G_1)\right).\\  
	\end{aligned}
	\end{equation*}		 
\end{proof}	
 The following corollary gives the adjacency characteristic polynomial of $G_1 \underline{\boxdot}G_2$ when $G_1$ and $G_2$ are both regular graphs.
\begin{cor}
	Let $G_i$ be an $r_i$-regular graph with $n_i$ vertices and $m_i$ edges for ${i=1,2}.$ Then the adjacency characteristic polynomial of $G_1 \underline{\boxdot}G_2$ is
	\begin{equation*}
	\begin{aligned}
	f(A(G_1 \underline{\boxdot}G_2),x) =&(x(x-r_2))^{m_1-n_1}\big[x^3+x^2(1-r_2+r_1-n_1)-x(2n_2r_1+2r_1+r_2+r_2r_1+\\&-n_1r_2)+2r_1r_2\big]\prod_{j=2}^{n_2}(x-\lambda_j(G_2))^{m_1} \prod_{i=2}^{n_1}\bigg[x^3+x^2(1-r_2+\lambda_i(G_1))\\&-x(n_2r_1+r_1+r_2+r_2\lambda_i(G_1)+n_2\lambda_i(G_1)+\lambda_i(G_1))+r_1r_2+r_2\lambda_i(G_1)\bigg].\\ 
	\end{aligned}
	\end{equation*} 	
\end{cor} 
The following corollary describes the complete spectrum of $G_1 \underline{\boxdot}G_2$, when $G_1$ and $G_2$ are both regular graphs.\\\\
\begin{cor}
	Let $G_i$ be an $r_i$-regular graph with $n_i$ vertices and $m_i$ edges for ${i=1,2}.$ Then the adjacency spectrum of $G_1 \underline{\boxdot}G_2$ consists of 
	\begin{enumerate}
		\item $0$, repeated $m_1-n_1$ times,
		\item $r_2$, repeated $m_1-n_1$ times,
		\item $\lambda_j(G_2),$ repeated $m_1$ times for $j=2,3,...,n_2$,
		\item three roots of the equation $x^3+x^2(1-r_2+\lambda_i(G_1))-x(n_2r_1+r_1+r_2+r_2\lambda_i(G_1)+n_2\lambda_i(G_1)+\lambda_i(G_1))+r_1r_2+r_2\lambda_i(G_1)=0  $ for $i= 2,...,n_1,$
		\item three roots of the equation $x^3+x^2(1-r_2+r_1-n_1)-x(2n_2r_1+2r_1+r_2+r_2r_1-n_1r_2)+2r_1r_2=0.$
	\end{enumerate}
\end{cor}
\begin{exam}
	Let $G_1=K_{3,3}$ and $G_2=K_2$. Then the adjacency  eigenvalues of $G_1$ are 0 (multiplicity $4$) and $\pm 3$, eigenvalues of $G_2$ are $\pm 1$. Therefore, the adjacency  eigenvalues of  $G_1 \underline{\boxdot}G_2$ are $0$ (multiplicity $3$), $1$ (multiplicity $3$), $-1$ (multiplicity $9$), roots of the equation $x^2-x-2=0$ (each root with multiplicity $3$), roots of the equation $x^3-3x^2-16x+6=0$, roots of the equation $x^3-3x^2+2x=0$ and roots of the equation $x^3-10x+3=0$ (each root with multiplicity $4$).\\
\end{exam}
Next, we shall consider the adjacency spectrum of $G_1 \underline{\boxdot}G_2$ when $G_1$ is a regular graph and $G_2=K_{p,q}$,  $G_2$ is non-regular if $p\ne q$.

\begin{cor}
	Let $G_1$ be an $r_1$-regular graph with $n_1$ vertices and $m_1$ edges . Then the adjacency spectrum of  $ G_1 \underline{\boxdot} K_{p,q}$ consists of
	\begin{enumerate}
		\item  $0$, repeated $m_1-n_1+m_1(p+q-2)$ times,
		\item $\pm\sqrt{pq}$, repeated $m_1-n_1$ times,
		\item four roots of the equation $x^4+(1+r_1-n_1)x^3-(pq+r_1p+r_1q+2r_1+pr_1+qr_1)x^2+(-pq-5r_1pq+n_1pq)x+2r_1pq=0$,
		\item four roots of the equation $x^4+(1+\lambda_i)x^3-(pq+r_1p+r_1q+r_1+(p+q+1)\lambda_i(G_1))x^2+(-pq-2r_1pq-3pq\lambda_i(G_1))x+r_1pq+pq\lambda_i(G_1)=0$ for $ i=2,...,n_1.$
		
	\end{enumerate}
\end{cor} The following corollary enables us to construct infinitely many pairs of $A$-cospectral graphs.
\begin{cor}
	$(a)$ Let $G_1$ and $G_2$ be $A$-cospectral regular graphs and $H$ is any regular graph. Then $ H \underline{\boxdot}  G_1$ and $ H \underline{\boxdot}  G_2$ are $A$-cospectral.\\
	$(b)$
	Let $G_1$ and $G_2$ be $A$-cospectral regular graphs and $H$ is any regular graph. Then $ G_1  \underline{\boxdot}  H$ and $ G_2  \underline{\boxdot}  H$ are $A$-cospectral.\\
	$(c)$ Let $G_1$ and $G_2$ be $A$-cospectral regular graphs, $H_1$ and $H_2$ are another $A$-cospectral regular graphs. Then $G_1  \underline{\boxdot}  H_1$ and $G_2  \underline{\boxdot}  H_2$ are $A$-cospectral.
\end{cor}
 Next, we consider the Laplacian characteristic polynomial  of $G_1 \underline{\boxdot}  G_2$. 
	\begin{thm}\label{th4.13}
	Let $G_1$ be an $r_1$-regular graph with $n_1$ vertices and $m_1$ edges and $G_2$ be an arbitrary graph with $n_2$ vertices. Then the Laplacian characteristic polynomial of $G_1 \underline{\boxdot}  G_2$ is 
		\begin{equation*}
	\begin{aligned}
	&f(L(G_1 \underline{\boxdot}  G_2),x) = (x-2)^{m_1-n_1}\Big[x^2-x(n_2r_1+r_1+2)\Big] \prod_{j=1}^{n_2}\Big(x-1-\mu_j(G_2)\Big)^{m_1}\\&\times\prod_{i=2}^{n_1}\Big[x^2-x(2+n_1+n_2r_1+\lambda_i(G_1))+2n_1+n_2r_1-r_1+\lambda_i(G_1)-n_2\lambda_i(G_1)\Big].\\  
	\end{aligned}
	\end{equation*} 
\end{thm}
\begin{proof}
	Let $L(G_2)$ be the Laplacian matrix of $G_2$. Then by a proper labeling of vertices, the Laplacian matrix of $G_1 \underline{\boxdot}  G_2$ can be written as
	\begin{align*}
	L(G_1 \underline{\boxdot}  G_2)=\begin{pmatrix}
	(n_1-1+n_2r_1)I_{n_1}-A(\bar {G_1})&-I(G_1)&-I(G_1) \otimes J_{1\times n_2}\\
	-I(G_1)^T&2I_{m_1}& O_{ m_1\times m_1n_2}\\
	-I(G_1)^T \otimes J_{1\times n_2}^T&O_{ m_1n_2\times m_1 }&I_{m_1}\otimes (2I_{n_2}+L(G_2)) 
	\end{pmatrix}. \\
	\end{align*}
	The Laplacian characteristic polynomial of $G_1 \underline{\boxdot}  G_2$ is \\
	\begin{align*}
	f(L(G_1 \underline{\boxdot}  G_2),x) =&  det\left (\begin{matrix}
	(x-n_1+1-n_2r_1)I_{n_1}+A(\bar {G_1})&I(G_1)&I(G_1) \otimes J_{1\times n_2}\\
	I(G_1)^T&(x-2)I_{m_1}& O_{ m_1\times m_1n_2}\\
	I(G_1)^T \otimes J_{1\times n_2}^T&O_{ m_1n_2\times m_1 }&I_{m_1}\otimes((x-2)I_{n_2}-L(G_2))  
	\end{matrix} \right).\end{align*} \\\textnormal{By Lemmas 2.1, 2.2, Definition 2.1, equation (2.0.1) and Corollary 2.6, we have}\\
	\begin{align*}
	f(L(G_1 \underline{\boxdot}  G_2),x)=& det(I_{m_1}\otimes((x-2)I_{n_2}-L(G_2)))det S,\\
	\textnormal{where~S} =&\begin{pmatrix}
	(x-n_1+1-n_2r_1)I_{n_1}+A(\bar {G_1})&{I(G_1)}\\
	I(G_1)^T&{(x-2)I_m}_1\\
	\end{pmatrix}\\&-\begin{pmatrix}
	I(G_1) \otimes J_{1\times n_2}\\
	O_{ m_1\times m_1n_2}
	\end{pmatrix}\left(I_{m_1}\otimes((x-2)I_{n_2}-L(G_2))\right)^{-1}\begin{pmatrix}
	I(G_1)^T \otimes J_{1\times n_2}^T&O_{ m_1n_2\times m_1 }\end {pmatrix}\\
	=&\begin{pmatrix}
	(x-n_1+1-n_2r_1)I_{n_1}+A(\bar {G_1})-\chi_{L{(G_2)}}(x-2)I(G_1)I(G_1)^T&I(G_1)\\
	I(G_1)^T&{(x-2)I_m}_1\\
	\end{pmatrix}.
	\end{align*}Therefore, det S 
	\begin{align*}
 	=&det\left (\begin{matrix}
	(x-n_1+1-n_2r_1)I_{n_1}+J_{n_1}-I_{n_1}-A( G_1)-\chi_{L{(G_2)}}(x-2)I(G_1)I(G_1)^T&I(G_1)\\
	I(G_1)^T&{(x-2)I_m}_1
	\end{matrix} \right)\\
	= & (x-2)^{m_1} det\left(x-n_1-n_2r_1)I_{n_1}+J_{n_1}-A( G_1)-\chi_{L{(G_2)}}(x-2)I(G_1)I(G_1)^T-\frac{I(G_1)I(G_1)^T}{x-2}\right)\\ 
	= & (x-2)^{m_1} det\left(x-n_1-n_2r_1)I_{n_1}+J_{n_1}-A( G_1)-\chi_{L{(G_2)}}(x-2)(A(G_1)+r_1I_{n_1})-\frac{A(G_1)+r_1I_{n_1}}{x-2}\right)
    \end{align*}
    \begin{equation*}\begin{small}{
    	\begin{aligned}
    	= & (x-2)^{m_1} det\left(\left(x-n_1-n_2r_1-r_1\chi_{L(G_2)}(x-2)-\frac{r_1}{x-2}\right)I_{n_1}+J_{n_1}-\left(1+\chi_{L(G_2)}(x-2)+\frac{1}{x-2}\right)A(G_1)\right).
    	\end{aligned}}\end{small}
    \end{equation*}
Note that $\mu_1(G_2)=0$, $P(\lambda_1(G_1))=n_1$,$P(\lambda_i(G_1))=0, i=2,...,n_1$ .\\
	Thus the characteristic polynomial of $L(G_1 \underline{\boxdot}  G_2)$	is	 
	\begin{equation*}
	\begin{aligned}
	&f(L(G_1 \underline{\boxdot}  G_2),x) = (x-2)^{m_1-n_1}\Big[x^2-x(n_2r_1+r_1+2)\Big] \prod_{j=1}^{n_2}\Big(x-1-\mu_j(G_2)\Big)^{m_1}\\&\times\prod_{i=2}^{n_1}\Big[x^2-x(2+n_1+n_2r_1+\lambda_i(G_1))+2n_1+n_2r_1-r_1+\lambda_i(G_1)-n_2\lambda_i(G_1)\Big].\\  
	\end{aligned}
	\end{equation*} 
\end{proof}	
The following corollary describes the complete Laplacian spectrum of $G_1 \underline{\boxdot}  G_2$, when $G_1$ is a regular graph and $G_2$ is an arbitrary graph.\\ 
\begin{cor}
	Let $G_1$ be an $r_1$-regular graph with $n_1$ vertices and $m_1$ edges and $G_2$ be an arbitrary graph with $n_2$ vertices. Then the Laplacian spectrum of $G_1 \underline{\boxdot}  G_2$ consists of 
	\begin{enumerate}
		\item  $2$, repeated $m_1-n_1$ times,
		\item  $1+\mu_j(G_2) $,  repeated $m_1$ times for $j=1,2,3,...,n_2,$
		\item two roots of the equation $x^2-x(n_2r_1+r_1+2)=0,$
		\item two roots of the equation $x^2-x(2+n_1+n_2r_1+\lambda_i(G_1))+2n_1+n_2r_1-r_1+(1-n_2)\lambda_i(G_1) =0 $ for $ i=2,3,...,n_1.$
	\end{enumerate}
\end{cor}
\begin{cor}
	Let $G_1$ be an $r_1$-regular graph with $n_1$ vertices and $m_1$ edges and $G_2$ be an arbitrary graph with $n_2$ vertices.
	Then the number of spanning trees of $G_1 \underline{\boxdot}  G_2$ is\\ $$t(G_1 \underline{\boxdot}  G_2)=\frac{2^{m_1-n_1}(n_2r_1+r_1+2)\prod_{j=1}^{n_2}(1+\mu_j(G_2))^{m_1}\prod_{i=2}^{n_1}\left[2n_1+n_2r_1-r_1+\lambda_i(G_1)-n_2\lambda_i(G_1)\right]}{n_1+m_1+m_1n_2}.$$
	
\end{cor}
\begin{cor}
	Let $G_1$ be an $r_1$-regular graph with $n_1$ vertices and $m_1$ edges and $G_2$ be  an arbitrary graph with $n_2$ vertices.
	Then the Kirchhoff index of $G_1 \underline{\boxdot}  G_2$ is\\\\ $Kf(G_1 \underline{\boxdot}  G_2)=(n_1+m_1+m_1n_2)\bigg[\frac{m_1-n_1}{2}+\frac{1}{n_2r_1+r_1+2}+\displaystyle\sum_{j=2}^{n_2}\frac{m_1}{1+\mu_j(G_2)}+\\~~~~~~~~~~~~~~~~~~~~~~~~~~~~~~~~~~~~~~~~~~~~~~~~~~~~~~~~~~~~~~~~~~~~~~~~``\displaystyle\sum_{i=2}^{n_1}\frac{2+n_1+n_2r_1+\lambda_i(G_1)}{2n_1+n_2r_1-r_1+\lambda_i(G_1)-n_2\lambda_i(G_1)}\bigg].$
	
\end{cor}
The following corollary helps us to construct infinitely many pairs of $L$-cospectral graphs.
\begin{cor}
	$(a)$ Let $G_1$ and $G_2$ be $L$-cospectral regular graphs and $H$ is an arbitarary graph. Then $H \underline{\boxdot}   G_1$ and $H \underline{\boxdot}   G_2$ are $L$-cospectral.\\
	$(b)$ 	Let $G_1$ and $G_2$ be $L$-cospectral regular graphs  and $H$ is an arbitrary  graph. Then $ G_1  \underline{\boxdot}   H$ and $ G_2  \underline{\boxdot}   H$ are $L$-cospectral.\\
	$(c)$ 	Let $G_1$ and $G_2$ be $L$-cospectral regular graphs, $H_1$ and $H_2$ are another $L$-cospectral regular graphs. Then $G_1  \underline{\boxdot}   H_1$ and $G_2  \underline{\boxdot}   H_2$ are $L$-cospectral.
\end{cor}	
Next we consider the signless Laplacian characteristic polynomial of $G_1\odot G_2$ when $G_1$ is regular and $G_2$ is an arbitrary graph.
\begin{thm}
	Let $G_1$ be an $r_1$-regular graph with $n_1$ vertices and $m_1$ edges and $G_2$ be an arbitrary graph with $n_2$ vertices.  Then the signless Laplacian characteristic polynomial of $G_1\underline{\boxdot} G_2$ is 
\begin{equation*}
\begin{aligned}
f(Q(G_1\underline{\boxdot} G_2),x) &= (x-2)^{m_1-n_1}\prod_{j=2}^{n_2}\Big(x-1-\gamma_j(G_2)\Big)^{m_1}\\&\times\prod_{i=1}^{n_1}\Bigg[\left(x-n_1+2-n_2r_1-r_1\chi_{Q(G_2)}(x-2)-\frac{r_1}{x-2}\right)-P(\lambda_i(G_1))+\\&~~~~~~~~~~~~~~~~~~~~~~~~~~~~~~~~~~(1-\chi_{Q(G_2)}(x-2)-\frac{1}{x-2})\lambda_i(G_1)\Bigg]\\  
\end{aligned}
\end{equation*}
\end{thm}

\begin{proof}
	Let $Q(G_2)$ be the signless Laplacian matrix of $G_2$. Then by a proper labeling of vertices, the Laplacian matrix of $G_1\underline{\boxdot} G_2$ can be written as
	\begin{align*}
	Q(G_1\underline{\boxdot} G_2)=\begin{pmatrix}
	(n_1-1+r_1n_2)I_{n_1}+A(\bar {G_1})&I(G_1)&I(G_1) \otimes J_{1\times n_2}\\
	I(G_1)^T&2I_{m_1}& O_{ m_1\times m_1n_2}\\
	I(G_1)^T \otimes J_{1\times n_2}^T&O_{ m_1n_2\times m_1 }&I_{m_1}\otimes (2I_{n_2}+Q(G_2)) 
	\end{pmatrix}. \\
	\end{align*}
	The signless Laplacian characteristic polynomial of $G_1\underline{\boxdot} G_2$ is \\
	\begin{align*}
	f(Q(G_1\underline{\boxdot} G_2),x) =&  det\left (\begin{matrix}
	(x-n_1+1-r_1n_2)I_{n_1}-A(\bar {G_1})&-I(G_1)&-I(G_1) \otimes J_{1\times n_2}\\
	-I(G_1)^T&(x-2)I_{m_1}&O_{ n_1\times m_1n_2} \\
	-I(G_1)^T \otimes J_{1\times n_2}^T&O_{ m_1n_2\times n_1 }&I_{m_1}\otimes((x-2)I_{n_2}-Q(G_2))  
	\end{matrix} \right).\end{align*} 
	The rest of the proof is similar to that of Theorem 5.6 and hence we omit details.
\end{proof}
Next we consider the signless Laplacian characteristic polynomial of $G_1\odot G_2$ when $G_1$  and $G_2$ are both regular graphs.
\begin{cor}
	Let $G_i$ be an $r_i$-regular graph with $n_i$ vertices and $m_i$ edges for ${i=1,2}.$ Then the signless Laplacian characteristic polynomial of $G_1\underline{\boxdot} G_2$ is 
	\begin{equation*}
	\begin{aligned}
	&f(Q(G_1\underline{\boxdot} G_2),x)\\ =& (x-2)^{m_1-n_1}\Big[x^3-x^2(n_2r_1+2r_2+2-r_1+2n_1)+x(2n_2r_1r_2-n_2r_1-2r_2r_1+4n_1r_2+3n_2r_1-6r_1+ \\&8n_1-4)-4n_2r_1r_2+6r_2r_1-8n_1r_2+2r_1r_2+6r_1-8n_1+2r_1+8r_2+8\Big] \\&\times \prod_{j=2}^{n_2}\Big(x-1-\gamma_j(G_2)\Big)^{m_1}\prod_{i=2}^{n_1}\Big[x^3-x^2(n_2r_1+2r_2+n_1+2-\lambda_i(G_1))+x(2n_2r_1r_2-n_2\lambda_i(G_1)\\&-2r_2\lambda_i(G_1)+2n_1r_2+3n_2r_1-5\lambda_i(G_1)+4n_1-r_1-4)+(-4n_2r_1r_2-2n_2\lambda_i(G_1)+6r_2\lambda_i(G_1))\\&\;\;\;\;\;\;\;\;\;\;\;\;\;\;\;\;\;\;\;\;\;\;\;\;\;\;\;\;\;\;\;\;\;\;\;\;\;\;\;\;\;\;\;\;\;\;\;\;\;-4n_1r_2-2n_2r_1+2r_1r_2+6\lambda_i(G_1)-4n_1+2r_1+8r_2+8 \Big].\\ 
	\end{aligned}
	\end{equation*}
\end{cor}
	The signless Laplacian spectrum of $G_1\underline{\boxdot}  G_2$ is obtained from the following corollary.
\begin{cor}
	Let $G_i$ be an $r_i$-regular graph with $n_i$ vertices and $m_i$ edges for ${i=1,2}.$ Then the signless Laplacian spectrum of $G_1\underline{\boxdot}  G_2$ consists of 
	\begin{enumerate}
		\item $2$, repeated $m_1-n_1$ times,
		\item  $1+\gamma_j(G_2) $,  repeated $m_1$ times for $ j=2,3,...,n_2,$
		\item three roots of the equation $x^3-x^2(n_2r_1+2r_2+2-r_1+2n_1)+x(2n_2r_1r_2-n_2r_1-2r_2r_1+4n_1r_2+3n_2r_1-6r_1+8n_1-4)-4n_2r_1r_2+6r_2r_1-8n_1r_2+2r_1r_2+6r_1-8n_1+2r_1+8r_2+8 =0,$
		\item three roots of the equation $x^3-x^2(n_2r_1+2r_2+n_1+2-\lambda_i(G_1))+x(2n_2r_1r_2-n_2\lambda_i(G_1)-2r_2\lambda_i(G_1)+2n_1r_2+3n_2r_1-5\lambda_i(G_1)+4n_1-r_1-4)+(-4n_2r_1r_2-2n_2\lambda_i(G_1)+6r_2\lambda_i(G_1))-4n_1r_2-2n_2r_1+2r_1r_2+6\lambda_i(G_1)-4n_1+2r_1+8r_2+8 =0$ for $i=2,...,n_1.$
	\end{enumerate}
\end{cor}
The following corollary helps us to construct infinitely many pairs of  $Q$-cospectral graphs.
\begin{cor}
	$(a)$	Let $G_1$ and $G_2$ be $Q$-cospectral regular graphs and $H$ is any regular graph. Then $H\underline{\boxdot}  G_1$ and $H\underline{\boxdot}  G_2$ are $Q$-cospectral.\\
	$(b)$ Let $G_1$ and $G_2$ be $Q$-cospectral regular graphs  and $H$ is any regular graph. Then $ G_1 \underline{\boxdot}  H$ and $ G_2 \underline{\boxdot}  H$ are $Q$-cospectral.\\
	$(c)$	Let $G_1$ and $G_2$ be $Q$-cospectral regular graphs, $H_1$ and $H_2$ are another $Q$-cospectral regular graphs. Then $G_1 \underline{\boxdot}  H_1$ and $G_2 \underline{\boxdot}  H_2$ are $Q$-cospectral.
\end{cor}
	\section{Conclusion}
In this paper, we determine the different types of spectra of central vertex corona, central edge corona, and central edge neighborhood corona of regular graphs. As an application of these results we construct infinitely many pairs of $A$-cospectral, $L$-cospectral and $Q$-cospectral graphs. Further, we compute the number of spanning trees and the  Kirchhoff index of these graphs. 
\bibliography{refecentral}

\begin{thebibliography}{10}

\bibitem{godsil2001algebraic}
Chris Godsil and Gordon~F Royle.
\newblock {\em Algebraic graph theory}, volume 207.
\newblock Springer Science \& Business Media, 2001.

\bibitem{klein1993resistance}
Douglas~J Klein and Milan Randi{\'c}.
\newblock Resistance distance.
\newblock {\em Journal of mathematical chemistry}, 12(1):81--95, 1993.

\bibitem{chen2007resistance}
Haiyan Chen and Fuji Zhang.
\newblock Resistance distance and the normalized {L}aplacian spectrum.
\newblock {\em Discrete applied mathematics}, 155(5):654--661, 2007.

\bibitem{gutman1996quasi}
Ivan Gutman and Bojan Mohar.
\newblock The quasi-{W}iener and the {K}irchhoff indices coincide.
\newblock {\em Journal of chemical information and computer sciences},
  36(5):982--985, 1996.

\bibitem{zhu1996extensions}
H-Y Zhu, Douglas~J Klein, and Istv{\'a}n Lukovits.
\newblock Extensions of the {W}iener number.
\newblock {\em Journal of chemical information and computer sciences},
  36(3):420--428, 1996.

\bibitem{harary1970corona}
Frank Harary and Roberto Frucht.
\newblock On the corona of two graphs.
\newblock {\em Aequationes mathematicae}, 4:322--325, 1970.

\bibitem{liu2013spectra}
Xiaogang Liu and Pengli Lu.
\newblock Spectra of subdivision-vertex and subdivision-edge neighbourhood
  coronae.
\newblock {\em Linear algebra and its applications}, 438(8):3547--3559, 2013.

\bibitem{lan2015spectra}
Jie Lan and Bo~Zhou.
\newblock Spectra of graph operations based on {R}-graph.
\newblock {\em Linear and Multilinear Algebra}, 63(7):1401--1422, 2015.

\bibitem{adiga2018spectra}
Ch~Adiga, BR~Rakshith, and KN~Subba~Krishna.
\newblock Spectra of some new graph operations and some new classes of integral
  graphs.
\newblock {\em Iranian Journal of Mathematical Sciences and Informatics},
  13(1):51--65, 2018.

\bibitem{adiga2015spectra}
Chandrashekar Adiga and BR~Rakshith.
\newblock Spectra of graph operations based on corona and neighborhood corona
  of the graphs and ${K}_1$.
\newblock {\em Journal of the international mathematical virtual institute},
  5:55--69, 2015.

\bibitem{horn1991topics}
Roger~A Horn and Charles~R Johnson.
\newblock Topics in matrix analysis. {C}ambridge university press.
\newblock {\em Cambridge, UK}, 1991.

\bibitem{cvetkovic1980spectra}
Drago{\v{s}}~M Cvetkovi{\'{c}}, Michael Doob, and Horst Sachs.
\newblock {\em Spectra of graphs, theory and applications}, volume~10.
\newblock Academic Press, New York, 1980.

\bibitem{das2018spectra}
Arpita Das and Pratima Panigrahi.
\newblock Spectra of {R}-vertex join and { R}-edge join of two graphs.
\newblock {\em Discussiones Mathematicae-General Algebra and Applications},
  38(1):19--31, 2018.

\bibitem{mcleman2011spectra}
Cam McLeman and Erin McNicholas.
\newblock Spectra of coronae.
\newblock {\em Linear algebra and its applications}, 435(5):998--1007, 2011.

\bibitem{liu19}
Xiaogang Liu and Zuhe Zhang.
\newblock Spectra of subdivision-vertex join and subdivision-edge join of two
  graphs.
\newblock {\em Bulletin of the Malaysian Mathematical Sciences Society},
  42(1):15--31, 2019.

\bibitem{vivin2008harmonious}
J~Vernold Vivin, MM~Akbar Ali, and K~Thilagavathi.
\newblock On harmonious coloring of central graphs.
\newblock {\em Advances and applications in discrete mathematics}, 2(1):17--33,
  2008.

\bibitem{van2003graphs}
Edwin~R Van~Dam and Willem~H Haemers.
\newblock Which graphs are determined by their spectrum?
\newblock {\em Linear Algebra and its applications}, 373:241--272, 2003.

\bibitem{jahfar2020central}
TK~Jahfar and AV~Chithra.
\newblock Central vertex join and central edge join of two graphs.
\newblock {\em AIMS Mathematics}, 5(6):7214--7233, 2020.

\end{thebibliography}
\bibliographystyle{unsrt}
\end{document}